\documentclass[10pt]{amsart}

\usepackage{amssymb,amsmath}
\usepackage{amsthm,mathabx}
\usepackage{qsymbols}
\usepackage{graphicx}

\usepackage[utf8]{inputenc}

\newtheorem{theorem}{Theorem}
\newtheorem{lemma}{Lemma}
\newtheorem{proposition}{Proposition}
\newtheorem{corollary}{Corollary}
\newtheorem{definition}{Definition}
\newtheorem{remark}{Remark}

\newcommand{\X}{\mathcal{X}}
\newcommand{\Y}{\mathcal{Y}}

\newcommand{\R}{\mathbb{R}}

\newcommand{\E}{\mathbb{E}}
\newcommand{\ep}{\varepsilon}

\newcommand{\Leb}{\mathrm{Leb}}
\newcommand{\IS}{\mathrm{IS}}

\usepackage{color}

\begin{document}

% Enter full title and short title for running headers
\title{The deficit in the Gaussian Log-Sobolev inequality and inverse Santal\'o inequalities}

\begin{abstract}
We establish dual equivalent forms involving relative entropy, Fisher information and optimal transport costs of inverse Santal\'o inequalities. We show in particular that the Mahler conjecture is equivalent to some dimensional lower bound on the deficit in the Gaussian logarithmic Sobolev inequality. We also derive from existing results on inverse Santal\'o inequalities some sharp lower bounds on the deficit in the Gaussian logarithmic Sobolev inequality. Our proofs rely on duality relations between convex functionals (introduced in \cite{CEK15} and \cite{San16}) related to the notion of moment measure.
\end{abstract}

\author{Nathael Gozlan}

\date{\today}

\thanks{The author is supported by a grant of the Simone and Cino Del Duca Foundation. This research has been conducted within the FP2M federation (CNRS FR 2036)}

\address{NG : Université de Paris, CNRS, MAP5 UMR 8145, F-75006 Paris, France}
\email{nathael.gozlan@u-paris.fr}


\maketitle

\maketitle

\section{Introduction}
The aim of this paper is to highlight some new connections between reverse forms of the Santal\'o inequality and some improved versions of the Gaussian logarithmic Sobolev inequality.  In particular, the celebrated Mahler conjecture is shown to be equivalent to some dimensional lower bound on the deficit in the logarithmic Sobolev inequality for the standard Gaussian measure. 

Recall the classical Santal\'o inequality \cite{San49}: if $K \subset \R^n$ is a convex body and 
\[
K^{\circ,z}:= \{ y  \in \R^n : (x-z) \cdot (y-z) \leq 1,\forall x \in K\}
\]
denotes its polar with respect to the point $z\in \R^n$ (simply denoted $K^\circ$ if $z=0$), then
\begin{equation}\label{eq:santalo-set}
P(K):= \inf_{z\in \R^n}\mathrm{Vol}(K)\mathrm{Vol}(K^{\circ,z}) \leq P(B_2^n),
\end{equation}
where $\mathrm{Vol}$ denotes the Lebesgue measure on $\R^n$ and, for any $p\geq 1$, $B_p^n = \{x\in \R^n : \sum_{i=1}^n |x_i|^p \leq 1\}$ denotes the $\ell_p$ unit ball of $\R^n.$
When $K$ is centrally symmetric, then the infimum in $P(K)$ is attained for $z =0$, and in this case, the Santal\'o inequality reads as follows
\[
\mathrm{Vol}(K)\mathrm{Vol}(K^\circ) \leq \mathrm{Vol}(B_2^n)^2.
\]
The Mahler conjecture \cite{Mahler39} states reverse bounds for $P(K)$, which are the following: 
if $K$ is centrally symmetric, then
\begin{equation}\label{eq:Mahler1}
\mathrm{Vol}(K)\mathrm{Vol}(K^\circ) \geq P(B_1^n) = \mathrm{Vol}(B_1^n)\mathrm{Vol}(B_{\infty}^n) = \frac{4^n}{n!}
\end{equation}
and for a general convex body $K$, 
\begin{equation}\label{eq:Mahler2}
P(K) \geq P(\Delta^n) = \frac{(n+1)^{n+1}}{(n!)^2}
\end{equation}
where $\Delta^n$ is any non-degenerate simplex of $\R^n$. Even if these two conjectures are still open, some progresses have been made in the understanding of this problem and some particular cases have been established. In  \cite{SR81}, Saint-Raymond (see also \cite{Mey86}) showed that \eqref{eq:Mahler1} holds true for unconditional convex bodies, that is to say convex body $K$ satisfying $x =(x_1,\ldots,x_n) \in K \Rightarrow (\ep_1x_1,\ldots,\ep_nx_n) \in K$, for all $\ep =(\ep_1,\ldots,\ep_n) \in \{-1,1\}^n.$ Other particular cases were established in \cite{R86, GMR88, Mey91, BF13, AFZ19}. Recently, Conjecture \eqref{eq:Mahler1} has been established in dimension $n=3$ by Iriyeh and Shibata (see \cite{FHMRPZ19} for an alternative proof).  
Bourgain and Milman \cite{BM87} (see also \cite{Kup08}, \cite{Naz12}, \cite{GPV14} and \cite{Ber20} for alternative proofs) showed that Conjecture \eqref{eq:Mahler2} is asymptotically true: there exists some absolute constant $\alpha>0$ such that for all $n\geq 1$ and all convex body $K\subset \R^n$, it holds 
\begin{equation}\label{eq:BM}
P(K) \geq \alpha^n P(\Delta^n).
\end{equation}
The Mahler conjectures admit functional equivalent versions that were considered in particular by Klartag and Milman \cite{KM05} and by Fradelizi and Meyer \cite{FM08a, FM08b}, that we shall now recall. 

We first need to introduce some notation and definitions that will be useful in all the paper.
We will denote by $\mathcal{F}(\R^n)$ the set of lower semi-continuous functions $f : \R^n \to \R\cup\{+\infty\}$ which are convex and such that $f(x)<+\infty$ for at least one value of $x$. The domain of a convex function $f$ is the convex set $\mathrm{dom}(f) = \{x \in \R^n : f(x) <+\infty\}$.
We recall, that the Fenchel-Legendre transform of $f\in \mathcal{F}(\R^n)$ is the function denoted by $f^*$ and defined by
\begin{equation}\label{eq:FL}
f^*(y) = \sup_{x\in \R^n}\{x\cdot y - f(x)\},\qquad y\in \R^n.
\end{equation}
A function $f:\R^n \to \R\cup\{+\infty\}$ is said unconditional if for any $\ep = (\ep_1,\ldots,\ep_n) \in \{-1,1\}^n$ it holds
\[
f(\ep_1x_1,\ldots,\ep_n x_n) = f(x_1,\ldots,x_n),\qquad \forall x = (x_1,\ldots,x_n) \in \R^n.
\]
We will denote by $\mathcal{F}_u(\R^n)$ the set of all unconditional elements of $\mathcal{F}(\R^n)$ and by $\mathcal{F}_s(\R^n)$ the set of functions $f \in \mathcal{F}(\R^n)$ that are symmetric: $f(-x)=f(x)$, $x\in \R^n$.
Finally, for any convex set $C \subset \R^n$, we will denote by $\chi_C$ the convex characteristic function of $C$ which is the function defined by $\chi_C(x) = 0$ if $x \in C$ and $+\infty$ otherwise. 
%\pagebreak

\begin{definition}[Functional Inverse Santal\'o Inequalities]\label{def:IS} Let $c>0$ and $n\geq 1$. \begin{itemize}
\item We will say that that the functional inverse Santal\'o inequality $\IS_n(c)$ holds with the constant $c>0$ if for all function $f\in \mathcal{F}(\R^n)$ such that $0<\int e^{-f}\,dx$ and $0<\int e^{-f^*}\,dx$, it holds
\begin{equation}\label{eq:InvS}
\int e^{-f}\,dx \int e^{-f^*}\,dx \geq c^n.
\end{equation}
\item We will say that that the \emph{symmetric} (resp. \emph{unconditional}) functional inverse Santal\'o inequality $\IS_{n,s}(c)$ (resp. $\IS_{n,u}(c)$) holds with the constant $c>0$ if  \eqref{eq:InvS} holds for all function $f \in \mathcal{F}_s(\R^n)$ (resp. $\mathcal{F}_u(\R^n)$) such that  $0<\int e^{-f}\,dx$ and $0<\int e^{-f^*}\,dx$. 
\end{itemize}
\end{definition}

Let us briefly recall how the functional and the convex body versions are related.  Let $K$ be a centrally symmetric convex body and denote by $\|x\|_K = \inf\{r \geq0 : x \in rK\}$, $x\in \R^n$, its gauge. Then an easy calculation shows that $\|\,\cdot\,\|_K^* = \chi_{K^\circ}$. Therefore $\int e^{-\|\,\cdot\,\|_K^*(x)}\,dx = \mathrm{Vol}(K^\circ).$ On the other hand,
\[
\int e^{- \|x\|_K}\,dx = \int_0^{+\infty} e^{-u} \mathrm{Vol} (\{x\in \R^n : \|x\|_K \leq u\})\,du =  \int_0^{+\infty} e^{-u} u^n\,du \mathrm{Vol}(K) = n!\mathrm{Vol}(K).
\]
Therefore, $\mathrm{IS}_{n,s}(4)$ implies \eqref{eq:Mahler1}. Conversely, it is shown in \cite[Proposition 1]{FM08a} that if \eqref{eq:Mahler1} holds for all $n\geq1$, then $\mathrm{IS}_{n,s}(4)$ holds for all $n\geq 1$. Furthermore,  according to \cite[Proposition 1]{FM08a} again, $\IS_n(e)$ holds for all $n\geq 1$ if and only if \eqref{eq:Mahler2} holds for all $n\geq 1$. Similarly, it follows from \eqref{eq:BM} that there exists some absolute constant $c>0$ such that $\IS_n(c)$ holds for all $n\geq 1$ (see \cite{KM05, FM08b}). In addition, Fradelizi and Meyer gave  in \cite{FM08b, FM08a} a direct functional proof of the fact that $\IS_{n,u}(4)$ holds for every $n\geq 1$, which gives back in particular Saint-Raymond's result. They also proved in \cite{FM08a} that $\IS_1(e)$ holds true (see also \cite{FM10}). Note that other special classes of functions are considered in \cite{FM08a,FM08b}.

The goal of this paper is to study dual forms, expressed on the space of probability measures, of the functional inverse Santal\'o inequality $\IS_n(c)$ and its variants. To state our main results, we need to introduce additional notations. We will denote by $\mathcal{P}(\R^n)$ the set of all Borel probability measures on $\R^n$, and by $\mathcal{P}_{k}(\R^n)$, $k\geq 1$, the subset of probability measures having a finite moment of order $k$. A probability measure $\nu \in \mathcal{P}(\R^n)$ realized by a random vector $X=(X_1,\ldots,X_n)$ will be said symmetric if $-X$ has the same law as $X$ and unconditional if $(\ep_1X_1,\ldots,\ep_nX_n)$ has the same law as $X$ for any $\ep \in \{-1,1\}^n.$ 
Finally, if $\nu_1,\nu_2 \in \mathcal{P}_k(\R^n)$, let us denote by $W_k(\nu_1,\nu_2)$ their Kantorovich transport distance of order $k$ (also called Wasserstein distance of order $k$), defined by
\[
W_k^k(\nu_1,\nu_2) = \inf \int |x-y|^k\,\pi(dxdy),
\]
where $|\,\cdot\,|$ denotes the standard Euclidean norm on $\R^n$ and where the infimum runs over the set of all transport plans $\pi$ between $\nu_1$ and $\nu_2$, that is to say the set of probability measures $\pi$ on $\R^n\times \R^n$ having $\nu_1$ and $\nu_2$ as marginals.

According to a celebrated result of Gross \cite{Gro75}, the standard Gaussian measure 
\[
\gamma_n (dx)= \frac{1}{(2\pi)^{n/2}} e^{-\frac{|x|^2}{2}}\,dx
\] 
on $\R^n$ satisfies the logarithmic Sobolev inequality: for all $\eta \in \mathcal{P}(\R^n)$ absolutely continuous with respect to $\gamma_n$,
\[
H(\eta |\gamma_n) \leq \frac{1}{2} I(\eta| \gamma_n), \qquad \forall \eta \in \mathcal{P}(\R^n),
\]
where, for any probability measure of the form $d\eta = h\,d\gamma_n$, the relative entropy $H(\eta |\gamma_n)$ of $\eta$ with respect to $\gamma_n$ is defined by 
\[
H(\eta | \gamma_n) = \int \log h \, d\eta.
\]
To define the Fisher information $I(\,\cdot\,| \gamma_n)$, we need to introduce additional material. We will say that a function $f:\R^n \to \R$ is absolutely continuous on almost every line parallel to an axis, it for every $i \in \{1,\ldots, n\}$ and Lebesgue almost every $(x_1,\ldots, x_{i-1},x_{i+1},\ldots, x_n) \in \R^{n-1}$, the function 
\[
t\mapsto f(x_1,\ldots, x_{i-1},t,x_{i+1},\ldots, x_n)
\]
is absolutely continuous on every segment. When $f$ satisfies this condition, its partial derivatives $\frac{\partial f}{\partial x_i}$, $i\in \{1,\ldots,n\}$, are defined Lebesgue almost everywhere. 
The Fisher information $I(\eta|\gamma_n)$ of a probability measure $d\eta=h\,d\gamma_n$ with respect to $\gamma_n$ is then defined by
\[
I(\eta| \gamma_n) = 4\int |\nabla (h^{1/2})|^2\,d\gamma_n,
\]
whenever $h^{1/2}$ is absolutely continuous on almost every line parallel to an axis, and $+\infty$ otherwise. It follows from \cite[Proposition 1.5.2]{Bog} and \cite[Chapter 1, Theorems 1 and 2]{Maz}, that $I(\eta| \gamma_n)<\infty$ if and only if $h^{1/2} \in W^{1,2}(\gamma_n)$ (the subspace of $L^2(\gamma_n)$ consisting of functions $f$ whose weak derivative is also in $L^2(\gamma_n)$), but we will not make reference to this space $W^{1,2}(\gamma_n)$ anymore in the paper.

\begin{remark}
If $h^{1/2}$ admits partial derivatives almost everywhere, the following quantity 
\begin{equation}\label{eq:tildeI}
\tilde{I}(\eta|\gamma_n) = 4\int |\nabla (h^{1/2})|^2\,d\gamma_n
\end{equation}
makes sense in $[0,\infty]$ and is such that $\tilde{I}(\eta|\gamma_n) \leq I(\eta|\gamma_n)$. Note however that the logarithmic Sobolev inequality is not always true if one replaces $I(\,\cdot\,|\gamma_n)$ by $\tilde{I}(\,\cdot\,|\gamma_n)$. Indeed, if for instance $d\eta=\frac{\mathbf{1}_{B}}{\gamma_n(B)}\,d\gamma_n$, where $B$ is (say) the Euclidean unit ball, then $0=\tilde{I}(\eta |\gamma_n)< I(\eta|\gamma_n)=+\infty$ whereas, $H(\eta |\gamma_n) = - \log \gamma_n(B) >0$.
\end{remark}

The deficit in the Gaussian logarithmic Sobolev inequality is the non-negative function $\delta_n$ defined by
\[
\delta_n (\eta) = \frac{1}{2} I(\eta| \gamma_n )- H(\eta |\gamma_n),
\]
for all $d\eta =h\,d\gamma_n$, such that $H(\eta |\gamma_n)<+\infty$.  Recently, bounding from below the function $\delta_n$ attracted a lot of attention. We refer to \cite{FMP13, IM14, BGRS14, DT16, FIL16, LNP17, CE17, BGG18, IK18, ELS19} and the references therein for some recent progresses regarding this question. 
The following theorem, which is one of our main results, shows in particular that the Mahler conjecture is equivalent to some particular bound on  $\delta_n$. 
%defined by
%\[
%\tilde{\delta}_n (\eta) = \frac{1}{2} \tilde{I}(\eta| \gamma_n )- H(\eta |\gamma_n),
%\]
% for all log-concave probability measure $\eta$ (note that in this case, $H(\eta |\gamma_n)$ is finite).
\begin{theorem}\label{thm:mainresult}
Let $c>0$ and $n\geq 1$. The inverse functional Santal\'o inequality $\IS_{n}(c)$ holds if and only if for all log-concave probability measures $\eta_1,\eta_2$ on $\R^n$ such that, for $i=1,2$, $d\eta_i = e^{-V_i}\,dx$ for some \emph{essentially continuous} $V_i \in \mathcal{F}(\R^n)$, it holds
\begin{equation}\label{eq:mainresult}
H(\eta_1 |\gamma_n)+ H(\eta_2 |\gamma_n) + \frac{1}{2}W_2^2(\nu_1,\nu_2) \leq \frac{1}{2} I(\eta_1| \gamma_n)+ \frac{1}{2} I(\eta_2| \gamma_n)+  n\log (2\pi/ c),
\end{equation}
as soon as $\nu_1,\nu_2 \in \mathcal{P}_2(\R^n)$, where, for $i=1,2$, $\nu_i = \nabla (V_i)_\# \eta_i$ is the moment probability measure of $\eta_i$.
Equivalently
\[
\delta_n(\eta_1) + \delta_n(\eta_2) \geq  \frac{1}{2}W_2^2(\nu_1,\nu_2)-n\log (2\pi/ c)
\]
or 
\[
\delta_{2n} (\eta_1 \otimes \eta_2) \geq  \frac{1}{2}W_2^2(\nu_1,\nu_2)-n\log (2\pi/ c).
\]
The same statement holds for $\IS_{n,s}(c)$ (resp. $\IS_{n,u}(c)$) with the extra condition that $\eta_1,\eta_2$ are symmetric (resp. unconditional).
\end{theorem}
Before commenting this result, we need to clarify some notions used in the statement above:
\begin{itemize}
\item An absolutely continuous measure $m$ (not necessarily finite) is said log-concave if $dm = e^{-V}\,dx$ for some $V: \R^n \to \R \cup\{+\infty\}$ convex (in this paper we don't consider log-concave measures supported on affine subspaces of dimension smaller than $n$).
\item A function $V\in \mathcal{F}(\R^n)$ is said to be essentially continuous if the set of points where it is discontinuous (as a function taking values in $\R\cup\{\infty\}$) is negligible for the Hausdorff measure $\mathcal{H}_{n-1}$. Equivalently, $V$ is essentially continuous if letting $D = \mathrm{dom}(V)$
\[
\mathcal{H}_{n-1}\left(\{ x \in \partial D : V(x) <\infty \}\right) = 0.
\]
Note in particular that in dimension $1$, a function $V\in \mathcal{F}(\R)$ is essentially continuous if and only if it is continuous as a function taking values in $\R \cup \{+\infty\}$. 
\item If $V \in \mathcal{F}(\R^n)$ is such that $0<\int e^{-V}<+\infty$, the moment measure of $V$ is the probability measure $\nu$ defined as the push forward of the probability measure $d\eta = \frac{e^{-V}}{\int e^{-V(y)}\,dy}dx$ under the map $\nabla V$.  By extension, we also say that $\nu$ is the moment measure of $\eta$. 
\item As explained in Remark \ref{rem:ac} below, if a probability measure is of the form $d\eta = e^{-V}\,dx$, with an  essentially continuous $V \in \mathcal{F}(\R^n)$, then its density $h$ with respect to $\gamma_n$ is such that $h^{1/2}$ is absolutely continuous on almost every line parallel to an axis. Note, for instance, that uniform distributions on convex bodies are never in this class.
\end{itemize}

According to the functional version of the Bourgain-Milman theorem established in \cite{KM05} and \cite{FM08b}, the inequality $\IS_n(c)$ holds true for some constant $c>0$ independent on $n$. We immediately conclude from this that for the same constant $c>0$ it holds for all $n\geq 1$
\begin{equation}\label{eq:deficitbound}
\delta_{2n} (\eta_1 \otimes \eta_2) \geq  \frac{1}{2}W_2^2(\nu_1,\nu_2)-n\log (2\pi/ c), 
\end{equation}
whenever $\eta_1,\eta_2$ are log-concave probability measures with an essentially continuous minus log density (and $\nu_1,\nu_2$ are the associated moment measures).
In dimension $1$, this result can be refined. Indeed, as we mentioned above, Fradelizi and Meyer \cite{FM08a} proved that $\IS_1(e)$ holds true. We thus derive from their result that \eqref{eq:deficitbound} holds true for $n=1$ and $c=e$. The following result shows that this bound on $\delta_2$ is sharp:
\begin{corollary}\label{cor:FMd1}
For all log-concave probability measures $\eta_1,\eta_2$ on $\R$ such that, for $i=1,2$, $d\eta_i = e^{-V_i}\,dx$ for some continuous convex function $V_i:\R\to \R\cup \{+\infty\}$, it holds
\[
\delta_{2} (\eta_1 \otimes \eta_2) \geq  \frac{1}{2}W_2^2(\nu_1,\nu_2)-\log (2\pi/ e),
\]
where, for $i=1,2$, $\nu_i = \nabla (V_i)_\# \eta_i$ is the moment probability measure of $\eta_i$.
This bound is equivalent to the functional inverse Santal\'o inequality $\IS_1(e)$. Moreover, there exist sequences of log-concave probability measures $(\eta_1^k)_{k\geq 1}$ and $(\eta_2^k)_{k\geq 1}$ with continuous densities as above (and with associated moment measures denoted by $\nu_1^k$, $\nu_2^k$, $k\geq 1$) such that
\[
\delta_{2} (\eta_1^k \otimes \eta_2^k) -  \frac{1}{2}W_2^2(\nu_1^k,\nu_2^k)+\log (2\pi/ e) \to 0
\]
as $k \to \infty$.
\end{corollary}
The sequences $(\eta_1^k)_{k\geq 1}$ and $(\eta_2^k)_{k\geq 1}$ are approximations in the class of log-concave measures with a continuous density of the following two probability measures
\[
\tau(dx) = e^{-(1+x)} \mathbf{1}_{[-1,+\infty[}(x)\,dx\qquad \text{and}\qquad \bar{\tau}(dx) = e^{x-1} \mathbf{1}_{]-\infty,1]}(x)\,dx
\]
whose minus log densities realize equality in $\IS_1(e)$, and are up to affine transformations the only cases of equality, as observed by Fradelizi and Meyer \cite{FM08a}. In particular, as the proof of Corollary \ref{cor:FMd1} will reveal, there is no equality cases in the logarithmic Sobolev formulation of the inverse Santal\'o inequality. This point will be further commented in Section \ref{sec:HWI}.

In a similar way, since $\IS_{n,u}(4)$ holds for every $n\geq 1$, the following result follows by choosing $\eta_2 = \tau_s^{\otimes n}$, where 
\[
\tau_s(dx) = \frac{1}{2} e^{-|x|}\,dx
\]
denotes the symmetric exponential distribution on $\R$. For every $n\geq 1$, let $C_n \subset \R^n$ be the unit discrete cube $C_n = \{-1,1\}^n$ and denote by $\lambda_{C_n}$ the uniform probability measure on $C_n$.

\begin{theorem}\label{thm:mainresult2}
For any log-concave and unconditional probability measure $\eta$ on $\R^n$ with $d\eta = e^{-V}\,dx$ where $V:\R^n \to \R\cup\{+\infty\}$ is an essentially continuous convex function, it holds
\[
H(\eta |\gamma_n)+\frac{1}{2}W_2^2\left(\nu,  \lambda_{C_n}\right)  \leq \frac{n}{2}\log \left(\frac{\pi e}{2}\right)+   \frac{1}{2} I(\eta| \gamma_n),
\]
where $\nu = \nabla V_\# \eta$ is the moment probability measure of $\eta$.
In other words, for such $\eta$,
\[
\delta_n(\eta) \geq \frac{1}{2}W_2^2\left(\nu,  \lambda_{C_n}\right)- \frac{n}{2}\log \left(\frac{\pi e}{2}\right).
\]
Moreover, there exists a sequence of product measures $(\eta_k^{\otimes n})_{k\geq 1}$ such that 
\[
\delta_n(\eta_k^{\otimes n}) - \frac{1}{2}W_2^2\left(\nu_k^{\otimes n},  \lambda_{C_n}\right) + \frac{n}{2}\log \left(\frac{\pi e}{2}\right) \to 0,
\]
as $k\to \infty$, where for $k\geq 1$, $\nu_k^{\otimes n}$ denotes the moment measure of $\eta_k^{\otimes n}$.
\end{theorem}
This time the sequence $(\eta_k)_{k\geq 1}$ is an approximation in the class of log-concave measures with a continuous density of the uniform measure on $[-1,1]$.  Note that Theorem \ref{thm:mainresult2} provides a new sharp dimensional lower bound on the deficit $\delta_n$ on the class of unconditional log-concave probability measures with a regular density.

Let us now give a flavor of the proof of Theorem  \ref{thm:mainresult} (in the case of $\IS_n(c)$, the other variants being similar). To prove Theorem \ref{thm:mainresult}, we will establish as an intermediate step that the reverse Santal\'o inequality $\IS_n(c)$ holds if and only if for all $\nu_1,\nu_2 \in \mathcal{P}_2(\R^n)$,
\begin{equation}\label{eq:intro1}
\inf_{\eta_1 \in \mathcal{P}_2(\R^n)} \left\{ \mathcal{T}(\nu_1,\eta_1) + H(\eta_1 | \mathrm{Leb})\right\} + \inf_{\eta_2 \in \mathcal{P}_2(\R^n)} \left\{ \mathcal{T}(\nu_2,\eta_2) + H(\eta_2 | \mathrm{Leb})\right\}  \leq -n\log c +\mathcal{T}(\nu_1,\nu_2),
\end{equation}
where $H(\,\cdot\, | \mathrm{Leb})$ denotes (minus) the Shannon entropy functional defined for all $d\eta = h\,dx$ by
\[
H(\eta | \mathrm{Leb}) = \int \log h \,d\eta
\] 
as soon as the integral makes sense, and where $\mathcal{T}(\,\cdot\,,\,\cdot\,)$ is the so-called maximal correlation transport cost defined as follows: for all $\nu_1,\nu_2\in \mathcal{P}_2(\R^n)$,
\[
\mathcal{T}(\nu_1,\nu_2) = \sup_{X \sim \nu_1, Y \sim \nu_2} \E[X\cdot Y].
\]
The proof of the equivalence between \eqref{eq:intro1} and $\IS_n(c)$ follows by adapting an argument of Bobkov and G\"otze \cite{BG99} showing equivalence between transport-entropy inequalities and infimum convolution inequalities (see also \cite{Gl10, GRST17} for extensions). While Bobkov and G\"otze argument was based on the classical duality relations between relative entropy and log-Laplace functionals (recalled in Section \ref{sec:classicalduality}), ours is based on a twisted duality involving the following functionals:
\[
L(f|\Leb) := - \log \int e^{-f^*}\,dx,\qquad f\in \mathcal{F}(\R^n).
\]
and
\[
K(\nu | \Leb) := \sup_{f \in L^1(\nu) \cap \mathcal{F}(\R^n)} \left\{\int (-f)\,d\nu -L(f|\Leb)\right\},\qquad  \nu \in \mathcal{P}_1(\R^n).
\] 
A simple calculation shows that 
\[
K(\nu | \Leb)  = - \inf_{\eta \in \mathcal{P}_1(\R^n)} \left\{ \mathcal{T}(\nu,\eta) + H(\eta | \mathrm{Leb})\right\}.
\]
To see that $\IS_n(c)$ implies \eqref{eq:intro1}, observe that  for all function $f\in \mathcal{F}(\R^n)$ such that $0<\int e^{-f}\,dx$ and $0<\int e^{-f^*}\,dx$ and $\nu_1,\nu_2 \in \mathcal{P}_2(\R^n)$, it holds
\[
\int -f\,d\nu_1 + \log \int e^{-f^*}\,dx + \int -f^*\,d\nu_2 + \log \int e^{-f}\,dx \geq n\log c-\left(\int f\,d\nu_1 +\int f^*\,d\nu_2\right).
\]
Bounding the left hand side by $K(\nu_1| \mathrm{Leb})+K(\nu_2| \mathrm{Leb})$, one sees that   \eqref{eq:intro1} follows (up to technicalities) by optimizing over $f$ and using the dual Kantorovich formula 
\[
\mathcal{T}(\nu_1,\nu_2) = \inf_{f\in \mathcal{F}(\R^n)} \int f\,d\nu_1 + \int f^*\,d\nu_2.
\]
Let us give an  idea of the proof of the converse implication. As observed by Cordero-Erausquin and Klartag \cite{CEK15}, a remarkable consequence of the Prekopa-Leindler inequality is that the functional $L(\,\cdot\,|\Leb)$ is convex on $\mathcal{F}(\R^n)$ (see the proof of Lemma \ref{lem:CEK} where this simple argument is recalled). The above functionals will be shown in Theorem \ref{thm:dualLeb} to be in convex duality (see Section \ref{Sec:Lebesgue} for precise statements about this duality), in the sense that the functional $L(\,\cdot\,|\Leb)$ can be recovered from the functional $K(\,\cdot\, | \Leb)$ as follows:
\[
L(f|\Leb) = \sup_{\nu \in \mathcal{P}_1(\R^n)} \left\{\int (-f)\,d\nu - K(\nu|\Leb)\right\}
\]
for all $f \in \mathcal{F}(\R^n)$ such that $\int e^{-f^*}\,dx>0$. This reverse relation is the key to complete the equivalence between $\IS_n(c)$ and \eqref{eq:intro1}.

To further analyze the inequality \eqref{eq:intro1}, we will make use of the remarkable characterization of moment measures recently obtained by Cordero-Erausquin and Klartag \cite{CEK15} (building on earlier works \cite{WZ04, Don08, BB13, Leg16})) and revisited by Santambrogio \cite{San16}. As shown in \cite{CEK15,San16}, for a given $\nu \in \mathcal{P}_1(\R^n)$ the quantity $\inf_{\eta \in \mathcal{P}_1(\R^n)} \left\{ \mathcal{T}(\nu,\eta) + H(\eta | \mathrm{Leb})\right\}$ is not $-\infty$ if and only if $\nu$ is centered and its support is not contained in a hyperplane (for completeness the proof of ``the only if" case is sketched in the proof of Proposition \ref{prop:domainK}). In this case, the optimal $\eta$ turns out to be a log-concave probability measure with a density of the form $e^{-V}$, where $V \in \mathcal{F}(\R^n)$ is an essentially smooth convex function and $\nu$ is the moment measure of $\eta$. The converse is also true: if $\nu$ is the moment measure of a given log-concave probability measure $\eta_o$ with a regular density as above, then the function $\eta \mapsto \mathcal{T}(\nu,\eta) + H(\eta | \mathrm{Leb})$ reaches its infimum at $\eta_o$. Let us mention that the notion of moment measures together with the above characterization recently found several applications in convex geometry \cite{Kla14, KlaKol17}, probability theory \cite{Fathi19,KK17} or functional inequalities \cite{FGJ17}. Here, we will use this description of moment measures to reparametrize the inequality \eqref{eq:intro1} in terms of $\eta_1,\eta_2$ instead of $\nu_1,\nu_2$, yielding to the following equivalent statement: for all log-concave probability measures $\eta_1,\eta_2$ with an essentially continuous log-density, it holds
\begin{equation}\label{eq:Entropy-Transport}
\mathcal{T}(\nu_1,\eta_1) + H(\eta_1 | \mathrm{Leb}) + \mathcal{T}(\nu_2,\eta_2) + H(\eta_2 | \mathrm{Leb})  \leq -n\log c +\mathcal{T}(\nu_1,\nu_2),
\end{equation}
where $\nu_1,\nu_2$ are the moment measures of $\eta_1,\eta_2$. This last inequality formulated with respect to the Lebesgue measure can then easily be recasted in terms of the Gaussian measure $\gamma_n$ yielding in particular to Theorem \ref{thm:mainresult}. 

Let us further comment the Entropy-Transport inequality \eqref{eq:Entropy-Transport}. It turns out that \eqref{eq:Entropy-Transport} also admits an information theoretic formulation. Recall that the entropy power of a random vector $X$ with law $\eta$ on $\R^n$ is defined as 
\begin{equation}\label{eq:entropypower}
N(X) = \frac{1}{2\pi e}\exp\left(-\frac{2}{n} H(\eta | \Leb)\right).
\end{equation}
With the notation above, one can easily prove (see Corollary \ref{cor:entropypower}) using a simple homogeneity argument that \eqref{eq:Entropy-Transport} is equivalent to 
\begin{equation}\label{eq:EPI}
N(X_1)N(X_2)\mathcal{T}(\nu_1,\nu_2)^2 \geq \left(\frac{nc}{2\pi}\right)^2,
\end{equation}
for random vectors $X_1,X_2$ having log concave distributions $\eta_1,\eta_2$ with full support and associated moments measures $\nu_1,\nu_2$. Let us note that if $X_1 \overset{d}{=} X_2$, then $\mathcal{T}(\nu_1,\nu_1) =  \int |\nabla V_1|^2\,d\eta_1 :=I(X_1)$ is the Fisher information of $\eta_1$. Indeed, the optimal coupling in $\mathcal{T}(\nu_1,\nu_1)$ is $(Y_1,Y_1)$ with $Y_1 \sim \nu_1$ so that 
\[
\mathcal{T}(\nu_1,\nu_1) = \int |x|^2\,\nu_1(dx) = \int |\nabla V_1(x)|^2\,\eta_1(dx) 
\]
So, in this case, \eqref{eq:EPI} boils down to 
\[
N(X_1)I(X_1) \geq \frac{nc}{2\pi}.
\]
A well known result of Stam \cite{Sta59} shows that the best constant in the inequality above is $c=2\pi$ (for general random vectors $X_1$). Inequality \eqref{eq:EPI} thus appears as some bivariate form of Stam's inequality for log-concave random vectors.

Before closing this introduction, let us point out that the results obtained in the present paper for reverse Santal\'o inequalities echo several preceding results developed in the framework of direct Santal\'o inequalities.
As proved by Ball in \cite{B87} in the case of even functions and then extended by Artstein-Avidan, Klartag and Milman \cite{AAKM04} and Fradelizi and Meyer \cite{FM07}, the direct Santal\'o inequality admits the following equivalent functional form: for any measurable function $f : \R^n \to \R\cup\{+\infty\}$, there exists $a \in \R^n$ such that
\begin{equation}\label{eq:directSantalo}
\int e^{-f_a}\,dx\int e^{-(f_a)^*}\,dx \leq (2\pi)^n,
\end{equation}
where $f_a(x) = f(x+a)$, $x\in \R^n$. When $f$ is even, $a$ can be chosen to be $0$. Direct proofs of this functional version were then obtained by Lehec \cite{Leh08, Leh09a, Leh09b}. The functional inequality \eqref{eq:directSantalo} immediately gives back the convex body version \eqref{eq:santalo-set}, but it is also interesting in itself. Let us mention two recent applications of the inequality \eqref{eq:directSantalo} that are of the same spirit as our main contributions. It was shown by Caglar, Fradelizi, Gu\'edon, Lehec, Sch\"utt and Werner \cite{CFGLSW16} that the inequality \eqref{eq:directSantalo} implies back some inverse logarithmic Sobolev inequality first obtained by Artstein-Avidan, Klartag, Sch\"{u}tt and Werner \cite{AKSW12}. More recently \cite{Fathi18}, Fathi showed that the inequality \eqref{eq:directSantalo} is in fact equivalent to some sharp symmetrized form of the Talagrand transport cost inequality (see Section \ref{sec:symTal} for more details). These symmetrized forms of Talagrand transport inequalities were further studied by Tsuji in \cite{Tsu20} (with in particular a direct transport proof of this sharp transport inequality in dimension 1). Finally, the inequality \eqref{eq:EPI} is reminiscent of a work by Lutwak, Yang and Zhang \cite{LYZ04} identifying the best constant $c_{p,\lambda,n}$ in the inequality
\[
c_{p,\lambda,n}(N_\lambda (X_1)N_\lambda(X_2))^{p/n} \leq \E[|X_1 \cdot X_2|^p]
\]
where $X_1,X_2$ are arbitrary independent random vectors on $\R^n$ with finite $p$-th moment, $N_\lambda$ is the $\lambda$-R\'enyi-entropy power, and the parameters $p,\lambda,n$ are in the range $p\geq 1$, $\lambda \geq \frac{n}{n+p}$. As proved in \cite{LYZ04}, this family of inequalities gives back the Santal\'o inequality when $X_1,X_2$ are uniformly distributed on convex bodies $K, K^\circ$ and when $\lambda$ and $p$ are sent to $\infty$.

The paper is organized as follows. In Section \ref{sec:duality}, we introduce two functionals $K(\,\cdot\, | m)$ and $L(\,\cdot\, | m)$ associated to a given log-concave measure $m$ on $\R^n$ (which coincide with the functionals considered above when $m$ is the Lebesgue measure). We study their basic properties and we show in Theorem \ref{thm:dualLeb} (using the result of \cite{CEK15}) that these functionals are convex conjugates when $m$ is the Lebesgue measure. This duality relation between these functionals turns out to be true for a general log-concave measure $m$, as shown in Theorem \ref{thm:dualitygen}. In Section 2, we use the duality between functionals 
$K(\,\cdot\, | \Leb)$ and $L(\,\cdot\, | \Leb)$ to establish several dual equivalent versions of the functional inverse Santal\'o inequality $\IS_{n}(c)$ and its variants. These dual versions involve various probability ``distances" such as (relative) entropy, (relative) Fisher information and optimal transport costs. The proofs of Theorems \ref{thm:mainresult} and \ref{thm:mainresult2} are given in this section. Finally, Section 3 contains the proof of Theorem \ref{thm:dualitygen} (based on Sion min-max theorem) and Section 4 an alternative proof of Theorem \ref{thm:dualLeb} (based on a general version of the Fenchel-Moreau biconjugation theorem).

\section{Duality results}\label{sec:duality}
In the following $m$ will always denote a Borel measure on $\R^n$ such that $m(K) <+\infty$ for all compact sets $K \subset \R^n.$

\subsection{Convex duality between relative entropy and log-Laplace functionals} \label{sec:classicalduality}
Consider the relative entropy functional with respect to $m$: for any $d\nu = h\,dm \in \mathcal{P}(\R^n)$ such that $\log h \in L^1(\nu)$
\[
H(\nu | m) = \int h\log h\,dm,
\]
with the usual convention $0 \log 0 = 0$.
\begin{remark}\label{rem:finite measure}
Note that, when $m$ is a finite measure, then the integral $\int h\log h\,dm$ always makes sense in $\R \cup \{+\infty\}$, since the function $x\log x$ is bounded from below. So in this case, we can extend the definition of $H(\,\cdot\,| m)$ by setting $H(\nu |m) = \int h\log h \,dm$ if $\nu \in \mathcal{P}(\R^n)$ is absolutely continuous with respect to $m$ and $d\nu = h\,dm$ and $H(\nu|m) = +\infty$ if $\nu$ is not absolutely continuous with respect to $m$. We will always adopt this convention when $m$ is a finite measure.
\end{remark}

Recall the following duality results for the relative entropy functional:
\begin{proposition}\label{prop:classicalduality}
If $d\nu = h\,dm \in \mathcal{P}(\R^n)$ with $\log h \in L^1(\nu)$, then
\begin{equation}\label{eq:DualH1}
H(\nu|m) = \sup\left\{ \int f\,d\nu - \log \int e^{f}\,dm : f \text{ s.t }  \int e^f\,dm <+\infty \right\}.
\end{equation}
and, if $\int e^f\,dm <+\infty$, then
\begin{equation}\label{eq:DualH2}
\log \int e^{f}\,dm = \sup\left\{ \int f\,d\nu - H(\nu|m) : d\nu = h\,dm \text{ with } \log h \in L^1(\nu)\right\}.
\end{equation}
\end{proposition}
In both formulas, $f$ is allowed to take values in $\R\cup\{\pm \infty\}$ and the fact that the integral $ \int f\,d\nu$ makes sense is a consequence of the proof below.
Equalities \eqref{eq:DualH1} and \eqref{eq:DualH2} express that the two convex functionals $\nu \mapsto H(\nu|m)$ and $f\mapsto \log \int e^f \,dm$ are in convex duality. For the sake of completeness, we recall a classical proof of these identities (see also \cite[Lemma 6.2.13]{DZ}).
\begin{proof}
Both results come from the following well known Young type inequality:
\[
xy \leq e^x + y\log y - y,\qquad \forall x \in \R, \forall y\geq0.
\]
Observe that if $d\nu = h\,dm \in \mathcal{P}(\R^n)$ with $\log h \in L^1(\nu)$ and $f$ is such that $\int e^f \,dm<+\infty$, then $f h \leq e^f  + h\log h - h$ and so $[fh]_+$ is $m$-integrable and satisfies 
\[
\int f\,d\nu \leq \int e^f\,dm + H(\nu|m)-1.
\]
Changing $f$ into $f+a$, for some $a \in \R$, then gives that
\[
\int f\,d\nu \leq e^a  \int e^f\,dm + H(\nu|m)-1 -a
\]
and optimizing over $a$ yields to 
\[
\int f\,d\nu \leq \log  \int e^f\,dm + H(\nu|m).
\]
For a given $\nu$, there is equality if $f = \log h$, whereas for a given $f$ such that $0<\int e^f\,dm<+\infty$, there is equality for $\nu = \frac{e^{f}}{\int e^f\,dm}$. If $\int e^f\,dm = 0$ (which means that $f=-\infty$ $m$ a.s), then it follows from the inequality above that $\int f\,d\nu = - \infty$ for any $d\nu = h\,dm$ such that $\log h \in L^1(\nu)$. 
This completes the proof.
\end{proof}

\subsection{A twisted log-Laplace functional}Following \cite{CEK15, San16}, we  will now consider a twisted version of  \eqref{eq:DualH1} and \eqref{eq:DualH2} where the Log-Laplace functional 
\[
f\mapsto \log \int e^f\,dm
\]
is replaced by the functional $L(\,\cdot\,|m)$ defined by
\begin{equation}\label{eq:Laplacestar}
L (f |m):= -\log \int e^{-f^*}\,dm,\qquad f\in \mathcal{F}(\R^n),
\end{equation}
where we recall that $f^*$ denotes the Fenchel-Legendre conjugate of $f$ defined in \eqref{eq:FL} and that $\mathcal{F}(\R^n)$ denotes the set of all convex and semicontinuous functions $f : \R^n \to \R \cup\{+\infty\}$, with a non empty domain. 

As observed in \cite{CEK15}, the functional $L(\,\cdot\,|m)$ turns out to be convex, when the measure $m$ is assumed to be log-concave. 
\begin{lemma}\label{lem:CEK}
If $m$ is a log-concave measure on $\R^n$, then for any measurable functions $f_0,f_1 : \R^n\to \R\cup\{+\infty\}$, it holds 
\[
 \int e^{-\left((1-t)f_0+tf_1\right)^*}\,dm \geq \left(\int e^{-f_0^*}\,dm\right)^{1-t}\left(\int e^{-f_1^*}\,dm\right)^t,\qquad \forall t \in [0,1].
\]
\end{lemma}
The proof of Lemma \ref{lem:CEK} given in \cite{CEK15} is a simple application of Prekopa's theorem (which is a particular case of the Prekopa-Leindler inequality). For completeness, we give below a slightly different derivation of Lemma \ref{lem:CEK} that we learned from an anonymous referee.
\begin{proof}
Let us write $dm=e^{-V}\,dx$ where $V$ is a convex function on $\R^n$. Observe that
\[
\Phi(t,x):=\left((1-t)f_0+tf_1\right)^*(x)+V(x) = \sup_{y\in \R^n} \{x\cdot y - (1-t)f_0(y)-tf_1(y)\}+V(x),\qquad x\in \R^n, t\in [0,1].
\]
Being a supremum of convex functions, $\Phi$ is a convex function on $[0,1]\times \R^n$. According to Prekopa's theorem \cite[Theorem 6]{Pre} on marginals of log-concave functions, the function $t \mapsto-\log \int_{\R^n} e^{-\Phi(t,x)}\,dx$ is also convex, which completes the proof.
\end{proof}
%\begin{proof}
%Since $m$ is log-concave, then as an immediate consequence of the Prekopa-Leindler inequality, it satisfies the following property: if $g_0,g_1,h:\R^n \to \R \cup\{+\infty\}$ are measurable functions such that for some $t\in (0,1)$ it holds 
%\[
%h((1-t)x+ty) \leq (1-t)g_0(x) + tg_1(y),\qquad \forall x,y \in \R^n
%\]
%then 
%\[
%\int e^{-h}\,dm \geq \left(\int e^{-g_0}\,dm\right)^{1-t}\left(\int e^{-g_1}\,dm\right)^{t}.
%\]
%Note that if $f_0,f_1$ are two measurable functions, then 
%\[
%\left((1-t)f_0+tf_1\right)^*((1-t)x+ty) \leq (1-t)f_0^*(x) +t f_1^*(y),\qquad \forall x,y \in \R^n,\qquad \forall t\in (0,1).
%\]
%So applying the inequality above to $g_0=f_0^*$, $g_1=f_1^*$ and $h= \left((1-t)f_0+tf_1\right)^*$ gives the result.
%\end{proof}

\subsection{A twisted version of the relative entropy functional} 
Mimicking \eqref{eq:DualH1}, we now introduce the following functional: for $\nu \in \mathcal{P}_1(\R^n)$,
\begin{align*}
K(\nu | m) &:= \sup_{f \in L^1(\nu) \cap \mathcal{F}(\R^n)} \left\{\int (-f)\,d\nu + \log \int e^{-f^*} \,dm\right\},\\
& = \sup_{f \in L^1(\nu) \cap \mathcal{F}(\R^n)} \left\{\int (-f)\,d\nu -L(f|m) \right\}.
\end{align*}
When $m$ is the Lebesgue measure on $\R^n$, we will use the notation $K(\,\cdot\,|\Leb)$.

This section is organized as follows: in Section \ref{sec:basicprop}, we first establish some basic properties of this functional, then we prove in Section \ref{sec:altexpr} an alternative expression for $K(\,\cdot\,|m)$ involving the maximum correlation transport cost $\mathcal{T}$ and finally, Section \ref{sec:revduality} establishes a reverse duality formula expressing back the functional $L(\,\cdot\,|m)$ in terms of $K(\,\cdot\,|m)$.

\subsubsection{Basic properties of the functional $K(\,\cdot\, | m)$}\label{sec:basicprop}

We will use repeatedly the following classical lemma in the sequel. For all $r\geq 0$, $B_r$ will denote in all the paper the closed Euclidean ball of radius $r\geq 0$ centered at the origin.
\begin{lemma}\label{lem:infconv}
If $f \in \mathcal{F}(\R^n)$, then for any $r>0$, the function $f_r$ defined by
\[
f_r(x) = \inf_{y\in \R^n} \{f(y) + r |x-y|\},\qquad x\in \R^n
\]
is convex, $r$-Lipschitz, satisfies $f_r \leq f$ and is such that $f_r^*(y) = f^*(y) + \chi_{B_r}(y)$, $y\in \R^n$.
Moreover $f_r \to f$ pointwise monotonically as $r \to +\infty$.
\end{lemma}
\begin{proof}
As an infimum of $r$-Lipschitz functions, $f_r$ is also $r$-Lipschitz. It clearly satisfies $f_r \leq f$ and is convex as an infimum convolution of two convex functions. The Legendre transform of $f_r$ can be calculated as follows:
\begin{align*}
f_r^*(y) & = \sup_{x\in \R^n} \{x\cdot y -  f_r(x) \}\\
& =  \sup_{x\in \R^n} \sup_{u\in \R^n}\{x\cdot y -  f(u)- r|x-u| \}\\
& =   \sup_{u\in \R^n}\left\{ u\cdot y -f(u) + \sup_{v\in \R^n}\{v\cdot y - r|v| \} \right\}\\
& = f^*(y) + \chi_{B_r}(y).
\end{align*}
For the pointwise convergence of $f_r$, we refer to \cite[Proposition 4.1.5]{HUL2}.
\end{proof}

It will often be useful to restrict the supremum defining $K(\,\cdot\, | m)$ to the smaller class $\mathcal{F}_{\mathrm{Lip}}(\R^n)$ of all convex and Lipschitz functions on $\R^n.$
\begin{proposition}\label{prop:Lip}
For any $\nu \in \mathcal{P}_1(\R^n)$, the supremum defining $K(\nu|m)$ can be restricted to $\mathcal{F}_{\mathrm{Lip}}(\R^n)$.
\end{proposition}

\begin{proof}
Consider $f_k(x) = \inf_{y\in \R^n} \{f(y) + k|x-y|\}$, $x \in \R^n$, as in Lemma \ref{lem:infconv}. It holds
\[
\int (-f)\,d\nu + \log \int e^{-f^*} \mathbf{1}_{B_k} \,dm \leq \int (-f_k)\,d\nu + \log \int e^{-f_k^*} \,dm \leq \sup_{g\in \mathcal{F}_{\mathrm{Lip}}(\R^n)} \left\{\int (-g)\,d\nu + \log \int e^{-g^*} \,dm\right\}.
\]
By monotone convergence, and optimizing over $f$, one concludes that 
\[
\sup_{f\in \mathcal{F}(\R^n) \cap L^1(\nu)}\left\{\int (-f)\,d\nu + \log \int e^{-f^*} \,dm \right\} \leq  \sup_{g\in \mathcal{F}_{\mathrm{Lip}}(\R^n)} \left\{\int (-g)\,d\nu + \log \int e^{-g^*} \,dm\right\}.
\]
The converse inequality being obvious, this completes the proof.
\end{proof}

Recall that the notions of symmetry and unconditionality were already defined in the Introduction for functions and for probability measures. Similarly, a measure $m$ on $\R^n$ (not necessarily of unit mass) is said unconditional if it is invariant under all flipping of coordinates: for all non-negative functions $h$ on $\R^n$ it holds
\[
\int h(\ep_1 x_1,\ldots,\ep_nx_n) m(dx_1,\ldots,dx_n) = \int h(x_1,\ldots,x_n) m(dx_1,\ldots,dx_n)
\]
for any $\ep = (\ep_1,\ldots,\ep_n) \in \{-1 ; 1\}^n$. We define similarly symmetric measure.
We will denote by $\mathcal{P}_{s,1}(\R^n)$ (resp. $\mathcal{P}_{u,1}(\R^n)$) the set of symmetric (resp. unconditional) elements of $\mathcal{P}_1(\R^n)$ and by $\mathcal{F}_{u}(\R^n)$ (resp. $\mathcal{F}_{u,\mathrm{Lip}}(\R^n)$) the subset of $\mathcal{F}(\R^n)$ consisting of unconditional functions (resp. Lipschitz and unconditional functions). We define similarly the sets $\mathcal{F}_{s}(\R^n)$ and $\mathcal{F}_{s,\mathrm{Lip}}(\R^n)$.
\begin{proposition}\label{prop:uncond}
If $m$ is log-concave and unconditional and $\nu \in \mathcal{P}_{u,1}(\R^n)$ (resp.  $\mathcal{P}_{s,1}(\R^n)$), the supremum defining $K(\nu | m)$ can be restricted to $\mathcal{F}_{u}(\R^n)$ or $\mathcal{F}_{u,\mathrm{Lip}}(\R^n)$ (resp. $\mathcal{F}_{s}(\R^n)$ or $\mathcal{F}_{s,\mathrm{Lip}}(\R^n)$).
\end{proposition}
\begin{proof} We only treat the unconditional case, the symmetric case being similar and simpler.
Let $\nu \in \mathcal{P}_1(\R^n)$ and $f \in \mathcal{F}(\R^n) \cap L^1(\nu)$. For any $\ep = (\ep_1,\ldots,\ep_n) \in \{-1 ; 1\}^n$, denote by $f_\ep$ the function defined by $f_{\ep} (x) = f(\ep_1 x_1,\ldots,\ep_nx_n)$, $x\in \R^n$, and by $\bar{f} \in \mathcal{F}_u(\R^n)$ the function defined by $\bar{f} =  \frac{1}{2^n} \sum_{\ep \in \{-1,1\}^n} f_\ep$. Since $(f_\ep)^* = (f^*)_\ep$ and the function $L(\,\cdot\,|m)$ is convex by Lemma \ref{lem:CEK}, it follows from the unconditionality of $\nu$ and $m$ that
\begin{align*}
\int (-f)\,d\nu -L(f |m)  & =  \int -\bar{f}\,d\nu -  \frac{1}{2^n} \sum_{\ep \in \{-1,1\}^n} L(f_\ep|m)\\
& \leq \int -\bar{f}\,d\nu - L(\bar{f} |m).
\end{align*}
Therefore, the supremum defining $K(\nu | m)$ can be restricted to $\mathcal{F}_{u}(\R^n)$. The same reasoning together with Proposition \ref{prop:Lip} shows that it can be further reduced to $\mathcal{F}_{u,\mathrm{Lip}}(\R^n)$. 
\end{proof}

Following \cite{CEK15}, we now collect some informations on the domain of $K(\,\cdot\,|\Leb)$:
\begin{proposition}\label{prop:domainK}
 A probability measure $\nu \in \mathcal{P}_1(\R^n)$ satisfies $K(\nu | \Leb)<+\infty$ if and only if $\int x\,\nu(dx)=0$ and the support of $\nu$ is not contained in a hyperplane.
\end{proposition}
\begin{proof}
We simply sketch the proof of the first implication.
Let $\nu \in \mathcal{P}_1(\R^n)$ and $f$ be a convex and Lipschitz function. Denoting $\ell_a(x) = a\cdot x$, $a,x\in \R^n$ and noticing that $(f+\ell_a)^*(y) = f^*(y-a)$, we get 
\[
K(\nu |\Leb) \geq  a\cdot \int x\,\nu(dx) - \int f\,d\nu + \log \int e^{-f^*(y-a)}\,dy =  a\cdot \int x\,\nu(dx) - \int f\,d\nu + \log \int e^{-f^*(y)}\,dy.
\]
So if $\int x\,\nu(dx) \neq 0$, then taking the supremum over $a$ gives that $K(\nu |\Leb) = +\infty$. 

Suppose now that the support of $\nu$ is included in a hyperplane $H$. Without loss of generality, one can assume that $H$ is the hyperplan $x_1 = 0.$ 
Let $f$ be the function defined by $f(x) = \chi_{\{0\}}(x_1) + \sum_{i=2}^n |x_i|$, $x\in \R^n$. Then, an easy calculation shows that $f^*(y) = \sum_{i=2}^n \chi_{[-1,1]}(y_i)$, $y\in \R^n$. Therefore, 
\[
\int e^{-f^*(y)}\,dy = \int \prod_{i=2}^n \mathbf{1}_{[-1,1]}(x_i)\,dx = +\infty.
\]
On the other hand, $\int f\,d\nu = \int  \sum_{i=2}^n |x_i| \,\nu(dx) <+\infty$, and so $K(\nu|m) = +\infty$.

The proof of the converse implication is much more involved. We refer to Proposition 12 of \cite{CEK15}.
\end{proof}

\subsubsection{An alternative expression}\label{sec:altexpr} In this paragraph, we assume that $m$ is a Borel measure on $\R^n$ such that 
\begin{equation}\label{eq:AssumptionInt}
\int e^{- \beta |x|} \,m(dx) <+\infty
\end{equation}
for some $\beta>0$.
This assumption is clearly satisfied for any log-concave measure on $\R^n.$ It will be convenient to introduce the probability measure $\bar{m}$ defined by $\bar{m}(dx) = \frac{e^{- \beta |x|}}{\int e^{- \beta |y|} \,m(dy) } m(dx).$

Under Assumption \eqref{eq:AssumptionInt}, one can unambiguously extend the definition of $H(\,\cdot\,| m)$ on the set $\mathcal{P}_1(\R^n)$, as follows:
\[
H(\nu| m) = \left\{\begin{array}{cc}  \int  \frac{d\nu}{dm}\log \frac{d\nu}{dm} \,dm & \text{if } \nu \ll m  \\ +\infty & \text{otherwise}  \end{array}\right.\qquad \forall  \nu \in \mathcal{P}_1(\R^n).
\]
To see that this definition makes sense, recall that according to Remark \ref{rem:finite measure}, the relative entropy $H(\nu | \bar{m})$ is well defined, for any $\nu \in \mathcal{P}(\R^n)$. Therefore, using that 
\[
H(\nu | \bar{m}) = H(\nu | m) + \beta \int |x|\,\nu(dx) + \text{constant},
\]
one sees that $H(\nu|m) = \int h\log h \,dm$ makes sense in $\R\cup\{\infty\}$ for any $d\nu = h\,dm \in \mathcal{P}_1(\R^n)$.

The functional $K(\,\cdot\, | m)$ admits another expression involving the so-called maximal correlation cost $\mathcal{T}$ that we shall now define. Given $\nu_1,\nu_2 \in \mathcal{P}_1(\R^n)$, we set
\[
\mathcal{T}(\nu_1,\nu_2) = \inf_{f \in \mathcal{F}(\R^n)} \left\{ \int f\,d\nu_1 + \int f^*\,d\nu_2 \right\}.
\]
Note that the integral of a convex function $f \in \mathcal{F}(\R^n)$ with respect to $\nu \in \mathcal{P}_1(\R^n)$ always makes sense in $\R\cup\{+\infty\}$ since, up to the subtraction of an affine function, $f$ can be assumed to be non-negative.
As already mentioned in the introduction, when $\nu_1,\nu_2 \in \mathcal{P}_2(\R^n)$, then it easily follows from the Kantorovich duality for the $W_2^2$ transport cost (see e.g \cite{Vil09}) that
\[
\mathcal{T}(\nu_1,\nu_2) = \sup \E[X\cdot Y],
\]
where the supremum runs over the set of pairs of random vectors $(X,Y)$ such that $X \sim \nu_1$ and $Y\sim \nu_2$.

\begin{proposition}\label{prop:altK}Under Assumption \eqref{eq:AssumptionInt}, for any $\nu \in \mathcal{P}_1(\R^n)$, it holds 
\[
K(\nu|m) = - \inf_{\eta \in \mathcal{P}_1(\R^n)} \left\{\mathcal{T}(\nu,\eta) + H(\eta | m)\right\},
\]
and the infimum can be restricted to compactly supported $\eta$'s. Moreover, if $\nu$ and $m$ are symmetric (resp. unconditional), then the infimum can be restricted to (compactly supported) elements of $\mathcal{P}_{s,1}(\R^n)$ (resp. $\mathcal{P}_{u,1}(\R^n)$).
\end{proposition}
\begin{proof}
By definition of $K(\,\cdot\,|m)$ and applying  Lemma \ref{lem:log-laplace} below to $\varphi = f^*$, one gets
\begin{align*}
K(\nu | m) & = \sup_{f \in L^1(\nu) \cap \mathcal{F}(\R^n)} \left\{\int (-f)\,d\nu + \log \int e^{-f^*} \,dm\right\}\\
& = \sup_{f \in L^1(\nu) \cap \mathcal{F}(\R^n)}  \sup_{\eta \in \mathcal{P}_1(\R^n)} \left\{ \int (-f)\,d\nu  + \int (-f^*)\,d\eta - H(\eta|m) \right\}\\
& = \sup_{\eta \in \mathcal{P}_1(\R^n)} \sup_{f \in L^1(\nu) \cap \mathcal{F}(\R^n)} \left\{ \int (-f)\,d\nu  + \int (-f^*)\,d\eta - H(\eta|m) \right\}\\
&= - \inf_{\eta \in \mathcal{P}_1(\R^n)} \left\{\mathcal{T}(\nu,\eta) + H(\eta | m)\right\}.
\end{align*}
Since the supremum in Lemma \ref{lem:log-laplace} can be restricted to compactly supported probability measures, the same is true for the infimum above.
\end{proof}

In the preceding proof, we used the following slight extension of the identity \eqref{eq:DualH2}.
\begin{lemma}\label{lem:log-laplace}For any $\varphi \in \mathcal{F}(\R^n)$, it holds 
\[
\log \int e^{-\varphi}\,dm = \sup_{\nu \in \mathcal{P}_1(\R^n)}\left\{ \int -\varphi\,d\nu - H(\nu|m)\right\},
\]
and the supremum can be restricted to compactly supported $\nu$. Moreover, if $\varphi$ and $m$ are symmetric (resp. unconditional), then the supremum can be restricted to (compactly supported) elements of $\mathcal{P}_{s,1}(\R^n)$ (resp. $\mathcal{P}_{u,1}(\R^n)$).
\end{lemma}
Note that, since $\varphi$ is convex, the integral $\int -\varphi\,d\nu$ makes sense in $\R\cup\{-\infty\}$ for any $\nu \in \mathcal{P}_1(\R^n).$
\begin{proof} Reasoning as in the proof of \eqref{eq:DualH1} and \eqref{eq:DualH2}, we see that if $\nu \in \mathcal{P}_1(\R^n)$ is such that $H(\nu|m) <\infty$ one has,
\[
\int -\varphi\,d\nu - H(\nu|m) \leq \log \int e^{-\varphi}\,dm
\]
and so taking the supremum over $\nu$, it holds
\[
\sup_{\nu \in \mathcal{P}_1(\R^n)}\left\{ \int -\varphi\,d\nu - H(\nu|m)\right\} \leq \log \int e^{-\varphi}\,dm.
\]
To show the converse inequality, consider $\nu_k(dx) = \frac{1}{Z_k}e^{-\varphi(x)}\mathbf{1}_{B_k}(x)\,m(dx)$, where we recall that $B_k$ is the closed ball of radius $k$ centered at $0$ and $Z_k = \int e^{-\varphi(x)}\mathbf{1}_{B_k}(x)\,m(dx).$ Since $\varphi$ is convex, there exists $a\in \R^n$, $b \in \R$ such that $\varphi(x) \geq a\cdot x + b$. The probability measure $\nu_k$ has  thus a bounded density and is supported on $B_k$, and so  belongs to $\mathcal{P}_1(\R^n).$ Also, $H(\nu_k | m) = \int_{B_k} -\varphi(x)e^{-\varphi(x)}\,dm - \log Z_k$, and the first integral is finite. Therefore, 
\[
\int -\varphi\,d\nu_k - H(\nu_k|m) = \log Z_k \to \log \int e^{-\varphi}\,dm 
\]
as $k\to \infty$, by monotone convergence. The fact that the supremum can be restricted to symmetric or unconditional $\eta$ when $\varphi$ and $m$ are symmetric or unconditional is left to the reader. 
This completes the  proof.
\end{proof}

\subsubsection{Reverse duality}\label{sec:revduality} The functional $K(\,\cdot\,|m)$ is defined as some sort of conjugate of the functional $L(\,\cdot\,|m)$.
In this paragraph, we address the question of the following reverse duality formula: 
\begin{equation}\label{eq:revduality}
\sup_{\nu \in \mathcal{P}_1(\R^n)} \left\{ \int (-f)\,d\nu - K(\nu |m)\right\}= L(f|m),\qquad f\in \mathcal{F}(\R^n),
\end{equation}
and we are looking for conditions on $f$ and $m$ under which \eqref{eq:revduality} holds true.

An easy observation, is that this formula always holds with $\leq$ instead of $=$, under no particular assumptions. 
\begin{proposition}\label{prop:dualtriv}
For any Borel measure $m$ on $\R^n$ and $f \in \mathcal{F}(\R^n)$, it holds 
\[
 \sup_{\nu \in \mathcal{P}_1(\R^n)} \left\{ \int (-f)\,d\nu - K(\nu |m)\right\} \leq L(f|m).
\]
\end{proposition}
In the following, a measure $m$ being fixed, we will denote by $\widetilde{\mathcal{F}}_{\mathrm{Lip}}(\R^n)$ the set of elements of $\mathcal{F}_{\mathrm{Lip}}(\R^n)$ such that $\int e^{-f^*}\,dm\neq 0.$
\begin{proof}
Let $f\in \mathcal{F}(\R^n)$; by definition of $K(\,\cdot\,|m)$, it holds 
\[
\sup_{\nu \in \mathcal{P}_1(\R^n)} \left\{ \int (-f)\,d\nu - K(\nu |m)\right\}  = \sup_{\nu \in \mathcal{P}_1(\R^n)} \inf_{\varphi \in \widetilde{\mathcal{F}}_{\mathrm{Lip}}(\R^n)} \left\{\int (\varphi-f)\,d\nu - \log \int e^{-\varphi^*}\,dm\right\}.
\]
Observe that, for any fixed $\nu \in \mathcal{P}_1(\R^n)$, it holds 
\[
\inf_{\varphi \in \widetilde{\mathcal{F}}_{\mathrm{Lip}}(\R^n)} \left\{\int (\varphi-f)\,d\nu - \log \int e^{-\varphi^*}\,dm\right\} \leq - \log \int e^{-f^*}\,dm.
\]
Indeed, defining $f_k(x) = \inf_{y\in \R^n}\{f(y) +k |x-y|\}$, $x\in \R^n, k\geq 1$, it follows from Lemma \ref{lem:infconv} that $f_k \leq f$, $f_k$ is $k$-Lipschitz, and $\int e^{-f_k^*}\,dm = \int e^{-f^*}\mathbf{1}_{B_k}\,dm$.
Therefore, for $k$ large enough $f_k \in \widetilde{\mathcal{F}}_{\mathrm{Lip}}(\R^n)$ and it holds
\[
\int (f_k-f)\,d\nu - \log \int e^{-f_k^*}\,dm \leq -\log \int e^{-f^*}\mathbf{1}_{B_k}\,dm,
\]
which letting $k \to \infty$ gives the claim.
\end{proof}

The following result shows that \eqref{eq:revduality} holds true at least when $m$ is the Lebesgue measure.
\begin{theorem}\label{thm:dualLeb}
For any $f\in \mathcal{F}(\R^n)$ such that $\int e^{-f^*}\,dx >0$, it holds
\[
\sup_{\nu \in \mathcal{P}_1(\R^n)} \left\{ \int (-f)\,d\nu - K(\nu |\mathrm{Leb})\right\}= L(f|\Leb),
\]
and the supremum can be restricted to compactly supported $\nu$. If $f$ is further assumed to be unconditional (resp. symmetric), then the supremum above can be restricted to unconditional $\mathcal{P}_{u,1}(\R^n)$ (resp. $\mathcal{P}_{s,1}(\R^n)$).
\end{theorem}
Below, we will derive Theorem \ref{thm:dualLeb} from the results of \cite{CEK15}. Another independent proof of Theorem \ref{thm:dualLeb} (based on a general Fenchel-Moreau biconjugation theorem) will be given in Section \ref{Sec:Lebesgue}.

We will need the following elementary lemma (also used in \cite{CEK15}):
\begin{lemma}\label{lem:elem} Let $\psi : \R^n \to \R\cup\{+\infty\}$ be some lower semicontinuous convex function such that $\int e^{-\psi}\,dx >0$. 
Then the following propositions are equivalent:
\begin{itemize}
\item[(i)] $\int e^{-\psi}\,dx <+\infty$,
\item[(ii)] There exists $a>0$ and $b \in \R$ such that $\psi(x) \geq a|x|+b$, $x\in \R^n$,
\item[(iii)] The point $0$ belongs to the interior of the set $\mathrm{dom} (\psi^*)$.
\end{itemize}
\end{lemma}
Observe that the lemma is no longer true if $\int e^{-\psi}\,dx=0$. For example, if $\psi = \chi_{H}$, for some hyperplan $H$, then (i) is true but (ii) is obviously false. Also, since $\psi^* = \chi_{H^\perp}$, (iii) is also false in this case. 
\begin{proof} It is clear that (ii) implies (i). The implication (i) $\Rightarrow$ (ii) is Lemma 2.1 of \cite{Kla07}.
To see that (ii) $\Rightarrow$ (iii), observe that for all $y \in \R^n$ such that $|y|\leq a$ it holds
\[
\psi^*(y) = \sup_{x\in \R^n} \{x\cdot y - \psi(x)\} \leq \sup_{x\in \R^n} \{x\cdot y - a|x|\} - b = \sup_{r\geq 0} \{r|y| - ar\}-b = - b,
\]
and so $0$ belongs to the interior of $\{x\in \R^n : \psi^*(x)<+\infty\}$.
Finally, let us show that (iii) implies (ii). Assume that there exists $a>0$ such that $\psi^*(y)<+\infty$ for all $y \in B_a$. Being convex, $\psi^*$ is continuous on $B_a$ and so there exists $b\in \R$ such that $\psi^* \leq -b + \chi_{B_a}$. Since $\psi$ is lower semicontinuous, one gets by duality that $\psi(x) \geq (-b+ \chi_{B_a})^* = b+ a|x|$, which completes the proof.
\end{proof}
Recall that the definitions of essentially continuous convex functions and of moment measures are given after Theorem \ref{thm:mainresult}.
\smallskip

\proof[Proof of Theorem \ref{thm:dualLeb}.]
Note that, according to Proposition \ref{prop:dualtriv}, there is nothing to prove if $\int e^{-f^*}\,dx=+\infty$. 

According to Theorem 8 of \cite{CEK15}, if $\psi_0,\psi_1 \in \mathcal{F}(\R^n)$ are such that $0<\int e^{-\psi_0}\,dx<+\infty$ and $0<\int e^{-\psi_1}\,dx<+\infty$ and $\psi_0$ is essentially smooth, then it holds 
\[
\log \int e^{-\psi_0}\,dx-\log \int e^{-\psi_1}\,dx\geq \int (\psi_0^* -\psi_1^*) d\nu_{\psi_0}.
\]
%(This inequality shows that, as soon as $f_0 \in \mathcal{F}(\R^n)$ is such that $0<\int e^{-f_0^*}\,dx<+\infty$ and $f_0^*$ is essentially smooth, the probability measure $\nu_{f_0^*}$ is a subgradient of $L(\,\cdot\, |\Leb)$ at the point $f_0$.)
In other words, 
\[
K(\nu_{\psi_0} | \mathrm{Leb}) = \int (-\psi_0^*)\,d\nu_{\psi_0} + \log  \int e^{-\psi_0}\,dx.
\]
Therefore, if $\psi_0$ is essentially continuous and $0<\int e^{-\psi_0}\,dx<+\infty$, then it holds
\[
\sup_{\nu \in \mathcal{P}_1(\R^n)} \left\{ \int (-\psi_0^*)\,d\nu - K(\nu |\mathrm{Leb})\right\} \geq \int (-\psi_0^*)\,d\nu_{\psi_0} - K(\nu_{\psi_0} |\mathrm{Leb}) = -\log \int e^{-\psi_0}\,dx
\]
and so, according to Proposition \ref{prop:dualtriv}, equality \eqref{eq:revduality} is satisfied for $f = \psi_0^*$. In other words, \eqref{eq:revduality} is true for any $f\in \mathcal{F}(\R^n)$ which is essentially continuous and such that $0<\int e^{-f^*}\,dx<+\infty$.

Now let us remove the assumption of essential continuity. Let $f\in \mathcal{F}(\R^n)$ be such that $0<\int e^{-f^*}\,dx<+\infty$ and let us prove \eqref{eq:revduality} in that case. Consider $f^{k}$ defined by
\[
f^{k} = f + \chi_{B_k},\qquad k\geq 1.
\]
Note that, according to Lemma \ref{lem:infconv}, 
\[
(f^k)^*(y) = (f^*)_k(y) = \inf_{x\in\R^n} \{f^*(x) + k |x-y|\},\qquad y\in \R^n.
\]
According to Lemma \ref{lem:elem}, since $\int e^{-f^*}\,dx<+\infty$ it follows that $0$ belongs to the interior of $\mathrm{dom}(f)$. Therefore, for any $k\geq 1$, $0$ also belongs to the interior of $\mathrm{dom}(f^k)$, and so $\int e^{-(f^k)^*}\,dx<+\infty$. Also, since $(f^k)^*$ is finite over $\R^n$, it is continuous on $\R^n$ and thus essentially continuous. Therefore, for every $k\geq 1$, it holds
\begin{align*}
-\log \int e^{-(f_k)^*}\,dx &= \sup_{\nu \in \mathcal{P}_1(\R^n)} \left\{ \int (-f-\chi_{B_k})\,d\nu - K(\nu |\mathrm{Leb})\right\}\\
& =\sup_{\nu \text{ compactly supported}} \left\{ \int (-f-\chi_{B_k})\,d\nu - K(\nu |\mathrm{Leb})\right\}.
\end{align*}
According to Lemma \ref{lem:infconv}, and the dominated convergence theorem (note that $e^{-(f_1)^*}$ is integrable), one gets
\begin{align*}
-\log \int e^{-f^*}\,dx &= \sup_{k\geq 1} -\log \int e^{-(f_k)^*}\,dx\\
& = \sup_{k\geq 1} \sup_{\nu \text{ compactly supported}} \left\{ \int (-f-\chi_{B_k})\,d\nu - K(\nu |\mathrm{Leb})\right\} \\
&=  \sup_{\nu \text{ compactly supported}}  \sup_{k\geq 1}  \left\{ \int (-f-\chi_{B_k})\,d\nu - K(\nu |\mathrm{Leb})\right\} \\
& =  \sup_{\nu \text{ compactly supported}} \left\{ \int (-f)\,d\nu - K(\nu |\mathrm{Leb})\right\}\\
& \leq \sup_{\nu \in \mathcal{P}_1(\R^n)} \left\{ \int (-f)\,d\nu - K(\nu |\mathrm{Leb})\right\}\\
& \leq -\log \int e^{-f^*}\,dx,
\end{align*}
where the last inequality comes from Proposition \ref{prop:dualtriv}. This completes the proof of the reverse duality formula.

Now let us assume that $f \in \mathcal{F}_u(\R^n)$ (the symmetric case is similar) and let us show that the supremum in the reverse duality formula can be restricted to $\mathcal{P}_{u,1}(\R^n)$. 
For any $\ep = (\ep_1,\ldots,\ep_n) \in \{-1 ; 1\}^n$ and $\nu \in \mathcal{P}_1(\R^n)$, denote by $\nu_\ep$ the push forward of $\nu$ under the map $x \mapsto (\ep_1 x_1,\ldots,\ep_nx_n)$, $x\in \R^n$, and consider the unconditional probability measure $\bar{\nu}=  \frac{1}{2^n} \sum_{\ep \in \{-1,1\}^n} \nu_\ep$. It is easily checked that $K(\nu_\ep | \Leb) = K(\nu | \Leb)$, for any $\ep \in \{-1 ; 1\}^n$. Therefore, $f$ being unconditional it holds
\begin{align*}
\int (-f)\,d\nu + K(\nu |\Leb) & =  \int -f\,d\bar{\nu} -  \frac{1}{2^n} \sum_{\ep \in \{-1,1\}^n} K(\nu_\ep| \Leb)\\
& \leq   \int -f\,d\bar{\nu} -  K(\bar{\nu}| \Leb),
\end{align*}
where the inequality follows from the convexity of $\nu \mapsto K(\nu | \Leb)$. This shows that the supremum in the reverse duality formula can be restricted to $ \mathcal{P}_{u,1}(\R^n)$. 
\endproof

It turns out that the conclusion of Theorem \ref{thm:dualLeb} can be extended to general log-concave measures $m$, as shown in the following result whose proof is postponed to Section \ref{sec:proof}.

\begin{theorem}\label{thm:dualitygen}
Suppose that $m$ is an arbitrary absolutely continuous log-concave measure.
For any $f\in \mathcal{F}(\R^n)$ such that $\int e^{-f^*}\,dm >0$, it holds
\[
\sup_{\nu \in \mathcal{P}_1(\R^n)} \left\{ \int (-f)\,d\nu - K(\nu |m)\right\}= L(f|m).
\]
If $m$ and $f$ are further assumed to be unconditional (resp. symmetric), then the supremum above can be restricted to $\mathcal{P}_{u,1}(\R^n)$ (resp. $\mathcal{P}_{s,1}(\R^n)$).
\end{theorem}

\section{HWI formulation of functional inverse Santal\'o inequalities}
In this section, we establish dual equivalent versions of the functional inverse Santal\'o inequalities introduced in Definition \ref{def:IS}. These equivalent versions are expressed in terms of entropy (H), Kantorovich transport distance (W) and Fisher information (I).
\begin{remark}
In \cite{FM08a}, \eqref{eq:InvS} is required to hold only for functions $f:\R^n \to \R$ such that $0<\int e^{-f}\,dx<+\infty$, without assumptions on $f^*$. Note that if $f$ satisfies these assumptions, then according to Lemma \ref{lem:elem}, the function $f^*$ is finite on a neighborhood of $0$, and therefore $\int e^{-f^*}\,dx>0$. It is not difficult to see that \eqref{eq:InvS} can then be extended to $f \in \mathcal{F}(\R^n)$ such that $0<\int e^{-f}\,dx$ and $0<\int e^{-f^*}\,dx$, so that the two definitions actually coincide.
\end{remark}

\subsection{Transport-Entropy form of reverse Santal\'o inequalities - Lebesgue version}
\begin{theorem}\label{thm:IS-T1} Let $c>0$. The reverse Santal\'o inequality $\IS_n(c)$ holds if and only if 
\begin{equation}\label{eq:InvS-Trans1}
K(\nu_1 | \mathrm{Leb}) + K(\nu_2 | \mathrm{Leb}) \geq n\log c - \mathcal{T}(\nu_1,\nu_2),
\end{equation}
for all $\nu_1,\nu_2 \in \mathcal{P}_1(\R^n)$ (resp. for all compactly supported $\nu_1,\nu_2$). In the case of the reverse Santal\'o inequality $\IS_{n,u}(c)$ (resp. $\IS_{n,s}(c)$), the same statement holds with the extra condition that $\nu_1,\nu_2$ belong to $\mathcal{P}_{u,1}(\R^n)$ (resp. $\mathcal{P}_{s,1}(\R^n)$).
\end{theorem}
\begin{proof} Assume that $\IS_n(c)$ holds and let us show \eqref{eq:InvS-Trans1}. Fix $\nu_1,\nu_2 \in  \mathcal{P}_1(\R^n)$. If $K(\nu_1 | \mathrm{Leb}) + K(\nu_2 | \mathrm{Leb}) = +\infty$ or $ \mathcal{T}(\nu_1,\nu_2) = +\infty$, there is nothing to prove. One can thus assume further that all these quantities are finite.
Let $f\in \mathcal{F}(\R^n)$ be such that $f\in L^1(\nu_1)$, $f^* \in L^1(\nu_2)$ (such $f$ exists since $\mathcal{T}(\nu_1,\nu_2)<+\infty$). Since for $i=1,2$, $K(\nu_i|\Leb) <+\infty$, Proposition \ref{prop:domainK} implies that $\nu_i$ is centered and that its support is not contained in a hyperplane. Therefore, $\overline{\mathrm{co}}(\mathrm{sup})(\nu_i)$ (the closed convex hull of the support of $\nu_i$) has a non empty interior.
Since $\int f\,d\nu_1<+\infty$, one easily concludes that 
\[
\overline{\mathrm{co}}(\mathrm{sup})(\nu_1) \subset \overline{\mathrm{dom} (f)}
\]
and so $f$ is finite on a small ball which implies that $\int e^{-f}\,dx>0$. Similarly $\int e^{-f^*}\,dx>0$. Applying the inequality \eqref{eq:InvS} then gives that 
\[
\int (-f)\,d\nu_1 + \log \int e^{-f^*}\,dx + \int (-f^*)\,d\nu_2 + \log \int e^{-f}\,dx \geq n\log c - \left(\int f \,d\nu_1 + \int f^*\,d\nu_2 \right).
\]
So, by definition of $K(\,\cdot\,|\mathrm{Leb})$, we get
\[
K(\nu_1 | \mathrm{Leb}) + K(\nu_2 | \mathrm{Leb}) \geq n\log c - \left(\int f \,d\nu_1 + \int f^*\,d\nu_2 \right).
\]
Optimizing over all $f \in \mathcal{F}(\R^n)$ such that $f\in L^1(\nu_1)$, $f^* \in L^1(\nu_2)$ yields
\[
K(\nu_1 | \mathrm{Leb}) + K(\nu_2 | \mathrm{Leb}) \geq n\log c - \mathcal{T}(\nu_1,\nu_2).
\]
Conversely assume that \eqref{eq:InvS-Trans1} holds for all compactly supported $\nu_1,\nu_2$. Take $f\in \mathcal{F}(\R^n)$ such that $0<\int e^{-f}\,dx$ and $0<\int e^{-f^*}\,dx$.
Since $\nu_1,\nu_2$ are compactly supported, $\mathcal{T}(\nu_1,\nu_2)$ is finite and it holds
\begin{align*}
\int (-f)\,d\nu_1 - K(\nu_1 | \mathrm{Leb}) + \int (-f^*)\,d\nu_2 - K(\nu_2 | \mathrm{Leb})  &\leq -n\log c - \left(\int f \,d\nu_1 + \int f^*\,d\nu_2 \right) + \mathcal{T}(\nu_1,\nu_2)\\
& \leq -n\log c,
\end{align*}
since by definition $\mathcal{T}(\nu_1,\nu_2) \leq \left(\int f \,d\nu_1 + \int f^*\,d\nu_2 \right)$, for any convex function $f$. Thus optimizing over all compactly supported $\nu_1,\nu_2$, it follows from Theorem \ref{thm:dualLeb} that
\[
-\log \int e^{-f^*}\,dx - \log \int e^{-f}\,dx  \leq -n \log c,
\]
which completes the proof. 
\end{proof}

The following is a straightforward consequence of Theorem \ref{thm:IS-T1} and Proposition \ref{prop:altK}.

\begin{corollary}\label{cor:IS-T1}
Let $c>0$. The reverse Santal\'o inequality $\IS_n(c)$ holds if and only if 
\begin{equation}\label{eq:InvS-Trans3}
\inf_{\eta_1 \in \mathcal{P}_1(\R^n)} \left\{ \mathcal{T}(\nu_1,\eta_1) + H(\eta_1 | \mathrm{Leb})\right\} + \inf_{\eta_2 \in \mathcal{P}_1(\R^n)} \left\{ \mathcal{T}(\nu_2,\eta_2) + H(\eta_2 | \mathrm{Leb})\right\}  \leq -n\log c +\mathcal{T}(\nu_1,\nu_2),
\end{equation}
for all $\nu_1,\nu_2 \in \mathcal{P}_1(\R^n)$ (resp. for all compactly supported $\nu_1,\nu_2$). In the case of the reverse Santal\'o inequality $\IS_{n,u}(c)$ (resp. $\IS_{n,s}(c)$), the same statement holds with the extra condition that $\nu_1,\nu_2,\eta_1,\eta_2$ belong to $\mathcal{P}_{u,1}(\R^n)$ (resp. $\mathcal{P}_{s,1}(\R^n)$).
\end{corollary}

%\pagebreak
We will now let moment measures enter the game using the following theorem.
\begin{theorem}[Cordero-Erausquin-Klartag/Santambrogio]\label{thm:CEKS}\
\begin{itemize}
\item[(i)] A probability measure $\nu \in \mathcal{P}(\R^n)$ is the moment measure of some log-concave probability measure $\eta_o$ on $\R^n$ such that $d\eta_o = e^{-V_o}\,dx$ for some essentially continuous convex function $V_o:\R^n\to \R\cup \{+\infty\}$ if and only if $\nu \in \mathcal{P}_1(\R^n)$, $\nu$ is centered and its support is not contained in a hyperplane. The function $V_o$ is moreover unique up to translations.
\item[(ii)] If $\nu$ is centered and its support is not contained in a hyperplane, then the probability measure $\eta_o$ is up to translations the unique minimizer of the functional $\eta \mapsto \mathcal{T}(\nu,\eta) + H(\eta | \Leb)$ on $\mathcal{P}_{1}(\R^n)$:
\[
\inf_{\eta \in \mathcal{P}_1(\R^n)}\{\mathcal{T}(\nu,\eta) + H(\eta | \Leb)\} = \mathcal{T}(\nu,\eta_o) + H(\eta_o | \Leb).
\]
\item[(iii)] Moreover, if $\nu \in \mathcal{P}_{u,1}(\R^n)$ (resp. $\mathcal{P}_{s,1}(\R^n)$) then $\eta_o \in \mathcal{P}_{u,1}(\R^n)$ (resp. $\mathcal{P}_{s,1}(\R^n)$).

\end{itemize}
\end{theorem}
In the preceding result, Item (i) is due to Cordero-Erausquin and Klartag  \cite{CEK15} and Item (ii) to Santambrogio \cite{San16}. Item (iii) is an immediate consequence of the second part of Proposition \ref{prop:altK}.

\begin{corollary}\label{cor:IS-T2}
Let $c>0$; the following propositions are equivalent:
\begin{itemize}
\item[(i)] Inequality $\IS_n(c)$ holds.
\item[(ii)] For all log-concave probability measures $\eta_1,\eta_2$ on $\R^n$ such that, for $i=1,2$, $d\eta_i = e^{-V_i}\,dx$ for some essentially continuous convex function $V_i:\R^n\to \R\cup \{+\infty\}$, it holds
\begin{equation}\label{eq:IS-T2}
\mathcal{T}(\nu_1,\eta_1) + H(\eta_1 | \Leb) + \mathcal{T}(\nu_2,\eta_2) + H(\eta_2 | \Leb) \leq  -n\log c +\mathcal{T}(\nu_1,\nu_2),
\end{equation}
where $\nu_1,\nu_2$ are the moment measures of $\eta_1$ and $\eta_2$. 
\item[(iii)] For all log-concave probability measures $\eta_1,\eta_2$ on $\R^n$ such that, for $i=1,2$, $d\eta_i = e^{-V_i}\,dx$ for some essentially continuous convex function $V_i:\R^n\to \R\cup \{+\infty\}$, it holds
\[
\int V_1^*\,d\nu_1 + \int V_2^* \,d\nu_2 \leq  -n\log c +\mathcal{T}(\nu_1,\nu_2),
\]
where $\nu_1,\nu_2$ are the moment measures of $\eta_1$ and $\eta_2$.
\end{itemize}
Moreover, if $V_i:\R^n\to \R$, then \eqref{eq:IS-T2} reduces to
\begin{equation}\label{eq:IS-T3}
 H(\eta_1 | \Leb) + H(\eta_2 | \Leb) \leq  -n\log (e^2c) +\mathcal{T}(\nu_1,\nu_2).
\end{equation}
The same result holds for inequality $\IS_{n,u}(c)$ (resp. $\IS_{n,s}(c)$) with the extra condition that $\eta_1,\eta_2$ are unconditional (resp. symmetric).
\end{corollary}

\begin{remark}\
\begin{itemize}
\item Note that the equivalence is still true if in (ii) one puts the extra condition that $\nu_1,\nu_2 \in \mathcal{P}_2(\R^n)$.
\item According to Lemma 5 of \cite{CEK15}, for a general $V_i \in \mathcal{F}(\R^n)$, the inequality 
\[
\mathcal{T}(\nu_i,\eta_i) \leq n
\] 
is always true. Therefore, for general $V_i$'s, \eqref{eq:IS-T3} is slightly stronger than \eqref{eq:IS-T2}.
\end{itemize}
\end{remark}

\begin{proof}
The equivalence between (i) and (ii) follows immediately from Corollary \ref{cor:IS-T1} and Theorem \ref{thm:CEKS}.

For $i=1,2$, let $\eta_i$ be a log-concave probability measure on $\R^n$ such that $d\eta_i = e^{-V_i}\,dx$, with $V_i \in \mathcal{F}(\R^n)$, and denote by $\nu_i$ the moment measure of $\eta_i$.
Let us show that (ii) and (iii) are equivalent. According to Proposition 7 of \cite{CEK15} and its proof, $\int |V_i| \,d\eta_i<+\infty$ and $\int |V_i^*|\,d\nu_i<+\infty$. Therefore, for any function $f\in\mathcal{F}(\R^n)$ such that $f \in L^1(\eta_i)$ and $f^* \in L^1(\nu_i)$, it follows from Young inequality that
\begin{align*}
\int f\,d\eta_i + \int f^*\,d\nu_i &= \int f(x) + f^*(\nabla V_i(x))\,\eta_i(dx)\\
& \geq \int x\cdot \nabla V_i(x)\,\eta_i(dx)\\
& = \int V_i(x) + V_i^*(\nabla V_i(x))\,d\eta_i\\
& =  \int V_i\,d\eta_i + \int V_i^*\,d\nu_i.
\end{align*}
Therefore, 
\[
\mathcal{T}(\nu_i,\eta_i) =  \int x\cdot \nabla V_i(x)\,\eta_i(dx) =  \int V_i\,d\eta_i + \int V_i^*\,d\nu_i.
\]
Since $H(\eta_i|\Leb) = -\int V_i\,d\eta_i$, we see that \eqref{eq:IS-T2} amounts to
\[
\int V_1^*\,d\nu_1 + \int V_2^* \,d\nu_2 \leq  -n\log c +\mathcal{T}(\nu_1,\nu_2).
\]
Now let us assume that $V_i:\R^n\to\R$ is finite over $\R^n$. Then 
\begin{equation}\label{eq:IPP}
\mathcal{T}(\nu_i,\eta_i) = \int x\cdot\nabla V_i(x) e^{-V_i(x)}\,dx = - \int \nabla\left(\frac{|x|^2}{2}\right)\cdot \nabla \left(e^{-V_i(x)} \right)\,dx = n,
\end{equation}
where the second equality follows by an integration by parts. This is clear if $V_i$ is continuously differentiable. For a general $V_i$, note that for any $j \in \{1,\ldots,n\}$ and for any fixed $x_1,\ldots,x_{j-1},x_{j+1},\ldots,x_n$ the function $x_j\mapsto x_j e^{-V_i(x_1,\ldots,x_{j-1}, x_j,x_{j+1},\ldots,x_n)}$ is locally Lipschitz and thus absolutely continuous. Therefore, for any $a>0$,
\begin{multline*}
ae^{-V_i(x_1,\ldots,x_{j-1}, a,x_{j+1},\ldots,x_n)} +a e^{-V_i(x_1,\ldots,x_{j-1}, -a,x_{j+1},\ldots,x_n)} 
= \int_{-a}^a e^{-V_i(x)}\,dx_j + \int_{-a}^a x_j\partial_j(V_i)(x)e^{-V_i(x)}\,dx_j.
\end{multline*}
Letting $a\to \infty$, integrating with respect to $x_1,\ldots,x_{j-1},x_{j+1},\ldots,x_n$ and summing over $j$ gives the result.
Therefore, when $V_i:\R^n\to \R$, \eqref{eq:IS-T2} is equivalent to 
\[
 H(\eta_1 | \Leb) + H(\eta_2 | \Leb) \leq  -n\log (e^2c) +\mathcal{T}(\nu_1,\nu_2).
\]
The cases of Inequalities $\IS_{n,u}(c)$ and $\IS_{n,s}(c)$ are straightforward.
\end{proof}

In the next result, we derive from \eqref{eq:IS-T3} an alternative formulation with an information-theoretic flavor. Recall the definition of the entropy power $N(X)$ given at \eqref{eq:entropypower}.
\begin{corollary}\label{cor:entropypower}
If $\IS_n(c)$ holds true then for any random vectors $X_1,X_2$ drawn according to log-concave distributions $\eta_1,\eta_2$ with full support on $\R^n$, it holds
\[
N(X_1) N(X_2) \mathcal{T}(\nu_1,\nu_2)^2 \geq \left(\frac{nc}{2\pi}\right)^2,
\]
where $\nu_1,\nu_2$ are the moment measures of $\eta_1,\eta_2$.
If $\IS_{n,s}(c)$ (resp. $\IS_{n,u}(c)$) holds true, then the inequality above holds with the extra condition that $X_1,X_2$ are symmetric (resp. unconditional).
\end{corollary}
\begin{proof} We only treat the case of the inequality $\IS_n(c)$ the other cases being similar.
Consider log-concave probability measures $d\eta_i=e^{-V_i}\,dx$, $i=1,2$, with $V_i : \R^n \to \R$ a finite valued convex function and let $X_i \sim \eta_i$. For any $\lambda>0$, define $\eta_i^\lambda$ as the pushforward of $\eta_i$ under the map $x\mapsto \lambda x$. Then $\eta_i^\lambda(dx) = e^{-V_i(x/\lambda)} \frac{1}{\lambda ^n}\,dx$, $i=1,2$, and so 
\[
H(\eta_i^\lambda | \Leb) = - n\log \lambda + H(\eta_i |\Leb).
\]
On the other hand, denoting $\nu_i^{\lambda}$ the moment measure of $\eta_i^\lambda$, then it is easily seen that $\nu_i^\lambda = \mathrm{Law} (\frac{1}{\lambda} \nabla V_i(X_i)).$ Therefore, 
\[
\mathcal{T}(\nu_1^\lambda,\nu_2^\lambda) = \frac{1}{\lambda^2} \mathcal{T}(\nu_1,\nu_2).
\]
So, according to \eqref{eq:IS-T3}, it holds
\[
H(\eta_1 |\Leb)+H(\eta_2 |\Leb) \leq n \log (\lambda^2) + \frac{1}{\lambda^2} \mathcal{T}(\nu_1,\nu_2) -n\log (e^2c).
\]
Optimizing over $\lambda$, yields to 
\[
H(\eta_1 |\Leb)+H(\eta_2 |\Leb) \leq n \log \left(\frac{\mathcal{T}(\nu_1,\nu_2)}{n}\right)  -n\log (ec),
\]
which completes the proof.
\end{proof}

\subsection{Transport-Entropy form of reverse Santal\'o inequalities - Gaussian version}\label{sec:symTal}
Recall that the standard Gaussian measure $\gamma_n$ on $\R^n$ satisfies the Talagrand transport-entropy inequality \cite{Tal96}:
\[
\frac{1}{2} W_2^2(\nu,\gamma_n) \leq H(\nu | \gamma_n),\qquad \forall \nu \in \mathcal{P}_2(\R^n).
\]
This inequality admits a symmetric version (which can be easily deduced from the one above using the triangle inequality for the distance $W_2$), which is the following:
\begin{equation}\label{eq:Talsym}
\frac{1}{4} W_2^2(\nu_1,\nu_2) \leq H(\nu_1 | \gamma_n)+H(\nu_2|\gamma_n),\qquad \forall \nu_1,\nu_2 \in \mathcal{P}_2(\R^n).
\end{equation}
The factor $1/4$ is sharp. Indeed, if one takes $\nu_1=\mathcal{N}(-a,1)$ and $\nu_2 = \mathcal{N}(a,1)$, for some $a>0$, then there is equality in \eqref{eq:Talsym}. Recently, it was shown by Fathi \cite{Fathi18} that the factor $1/4$ can be improved to $1/2$ if at least one of the measures $\nu_1,\nu_2$ is centered. This result is a consequence of the functional form of the Santal\'o inequality. Below, we show that reverse Santal\'o Inequalities can be translated in terms of lower bounds for the following functional
\[
\mathcal{G}(\nu_1,\nu_2) = H(\nu_1 | \gamma_n)+ H(\nu_2 | \gamma_n) - \frac{1}{2}W_2^2 (\nu_1,\nu_2),\qquad \forall \nu_1,\nu_2 \in \mathcal{P}_2(\R^n).
\]

\begin{theorem}\label{thm:InvS-Trans}
Let $c>0$. The reverse Santal\'o inequality $\IS_n(c)$ holds if and only if for all $\nu_1,\nu_2 \in  \mathcal{P}_2(\R^n)$, it holds
\begin{equation}\label{eq:InvS-Trans2}
\mathcal{G}(\nu_1,\nu_2) \geq \inf_{\eta_1 \in  \mathcal{P}_2(\R^n)} \mathcal{G}(\eta_1,\nu_2) + \inf_{\eta_2 \in  \mathcal{P}_2(\R^n)} \mathcal{G}(\nu_1,\eta_2)  +n\log (c/(2\pi)).
\end{equation}
In the case of the reverse Santal\'o inequality $\IS_{n,u}(c)$ (resp. $\IS_{n,s}(c)$), the same statement holds with the extra condition that $\nu_1,\nu_2,\eta_1,\eta_2$ belong to $\mathcal{P}_{u,2}(\R^n)$ (resp. $\mathcal{P}_{s,2}(\R^n)$).
\end{theorem}
\begin{proof}
We only treat the case of the inequality $\IS_n(c)$, the others being similar.  According to Corollary \ref{cor:IS-T1}, the inequality $\IS_n(c)$ is equivalent to 
\[
\inf_{\eta_1 \in \mathcal{P}_2(\R^n)} \left\{ \mathcal{T}(\nu_1,\eta_1) + H(\eta_1 | \mathrm{Leb})\right\} + \inf_{\eta_2 \in \mathcal{P}_2(\R^n)} \left\{ \mathcal{T}(\nu_2,\eta_2) + H(\eta_2 | \mathrm{Leb})\right\}  \leq -n\log c +\mathcal{T}(\nu_1,\nu_2),
\]
for all $\nu_1,\nu_2 \in \mathcal{P}_2(\R^n)$ (we could even restrict $\eta_1,\eta_2,\nu_1,\nu_2$ to compactly supported probability measures).

If $\nu_1,\nu_2 \in \mathcal{P}_2 (\R^n)$, then
\begin{equation}\label{eq:TW1}
\mathcal{T}(\nu_1,\nu_2)= - \frac{1}{2}W_2^2(\nu_1,\nu_2) + \frac{1}{2}\int |x|^2 \,d\nu_1 + \frac{1}{2}\int |x| ^2\,d\nu_2 
\end{equation}
and, if $\eta_1,\eta_2 \in \mathcal{P}_2 (\R^n)$, then
\begin{equation}\label{eq:TW2}
\mathcal{T}(\nu_i,\eta_i)= - \frac{1}{2}W_2^2(\nu_i,\eta_i) + \frac{1}{2}\int |x|^2 \,d\nu_i + \frac{1}{2}\int |x| ^2\,d\eta_i. 
\end{equation}
Also, note that
\begin{equation}\label{eq:HLG}
H(\eta_i | \gamma_n) = H(\eta_i | \mathrm{Leb}) - \int \log \frac{d\gamma_n}{dx}\,d\eta_i = H(\eta_i | \mathrm{Leb}) +  \frac{1}{2}\int |x| ^2\,d\eta_i + \frac{n}{2}\log (2\pi).
\end{equation}
So, we get 
\[
\inf_{\eta_1 \in \mathcal{P}_2(\R^n)} \left\{ - \frac{1}{2}W_2^2(\nu_1,\eta_1)+ H(\eta_1 | \gamma_n)\right\} + \inf_{\eta_2 \in \mathcal{P}_2(\R^n)} \left\{- \frac{1}{2}W_2^2(\nu_2,\eta_2)+ H(\eta_2 | \gamma_n)\right\}  \leq n\log (2\pi/ c) - \frac{1}{2}W_2^2(\nu_1,\nu_2).
\]
So adding $H(\nu_1 | \gamma_n)+ H(\nu_2 | \gamma_n)$, gives the claim.
\end{proof}

\subsection{The deficit in Log-Sobolev and reverse Santal\'o inequalities}\label{sec:HWI} We are now ready to prove our main result (Theorem \ref{thm:mainresult}) which gives an equivalent formulation of functional inverse Santal\'o inequalities in terms of the deficit in the Gaussian logarithmic Sobolev inequality. 

\smallskip

\proof[Proof of Theorem \ref{thm:mainresult}.] Again we only treat the case of the inequality $\IS_n(c)$, the other being similar.
According to Corollary \ref{cor:IS-T2}, the inequality $\IS_n(c)$ holds if and only if for all log-concave measures $\eta_1,\eta_2$ satisfying the assumptions of the theorem, the inequality \eqref{eq:IS-T2} holds true.
Assuming that $\nu_1,\nu_2 \in \mathcal{P}_2(\R^n)$ and using \eqref{eq:TW1}, \eqref{eq:TW2}, \eqref{eq:HLG}, one sees that \eqref{eq:IS-T2} amounts to 
\begin{equation}\label{eq:InvSequiv}
-\frac{1}{2}W_2^2(\nu_1,\eta_1) + H(\eta_1 | \gamma_n) -\frac{1}{2} W_2^2(\nu_2,\eta_2) + H(\eta_2 | \gamma_n) \leq  n\log (2\pi/c) -\frac{1}{2}W_2^2(\nu_1,\nu_2).
\end{equation}
Note that
\[
h_i(x):=\frac{d \eta_i}{d\gamma_n}(x) = (2\pi)^{n/2}e^{-(V_i(x)- \frac{|x|^2}{2})},\qquad \forall x\in \R^n.
\]
This function is differentiable almost everywhere, so
\[
\tilde{I}(\eta_i | \gamma_n) = \int \frac{|\nabla h_i|^2}{h_i}\,d\gamma_n(x) =  \int |\nabla V_i(x)-x|^2 e^{-(V_i(x)- \frac{|x|^2}{2})} e^{- \frac{|x|^2}{2}}\,dx,
\]
where $\tilde{I}$ is defined in \eqref{eq:tildeI}.
On the other hand,
\[
W_2^2(\nu_i,\eta_i) = \int |\nabla V_i(x)-x|^2 e^{-(V_i(x)-  \frac{|x|^2}{2})} e^{- \frac{|x|^2}{2}}\,dx
\]
and so $\tilde{I}(\eta_i | \gamma_n) = W_2^2(\nu_i,\eta_i)$.

Therefore, \eqref{eq:InvSequiv} is equivalent to 
\[
H(\eta_1 |\gamma_n) - \frac{1}{2} \tilde{I}(\eta_1| \gamma_n) + H(\eta_2 |\gamma_n) - \frac{1}{2} \tilde{I}(\eta_2| \gamma_n) \leq n\log (2\pi/ c) - \frac{1}{2}W_2^2(\nu_1,\nu_2).
\]
Since $ \tilde{I}(\eta_i| \gamma_n) \leq  I(\eta_i| \gamma_n)$, $i=1,2$, this completes the proof.
\endproof
 
 \begin{remark}\label{rem:ac}
If $d\eta = e^{-V}\,dx$ with an essentially continuous $V \in \mathcal{F}(\R^n)$, then $\tilde{I}(\eta| \gamma_n)=I(\eta| \gamma_n)$. 
Indeed, according to Lemma \ref{lem:ac} below, the function $e^{-V/2}$ is absolutely continuous on almost every line parallel to an axis and so the same is true for $h^{1/2}$, where $h(x)=(2\pi)^{n/2}e^{-(V(x)- \frac{|x|^2}{2})}$, $x\in \R^n$, is the density of $\eta$ with respect to $\gamma_n$. In particular, we don't loose anything in the last step of the proof of Theorem \ref{thm:mainresult}.
 \end{remark}

\begin{lemma}\label{lem:ac}
Let $W \in \mathcal{F}(\R^n)$ be essentially continuous. Then the function $f=e^{-W}$ is absolutely continuous on almost every line parallel to an axis.
\end{lemma}
\begin{proof}
First let us show the lemma in dimension $n=1$. Let $W \in \mathcal{F}(\R)$ be essentially continuous and not identically $+\infty$ (otherwise there is nothing to prove), and let us show that $f=e^{-W}$ is absolutely continuous on any segment. To fix the idea, one can assume that $\mathrm{dom}(W) = (a,+\infty)$ with $W(x) \to +\infty$ as $x \to a$ (the other cases are similar). According to e.g \cite[Theorem 1.1.9]{Hor}, for any $\varepsilon>0$, the function $W$ satisfies,
\[
W(x) - W(a+\varepsilon) = \int_{a+\varepsilon}^{x} W'_r(u)\,du,\qquad \forall x \in [a+\varepsilon, +\infty),
\]
where $W'_r$ denotes the right derivative of $W$ (which is well defined on $(a,+\infty)$). Fix $b>a$ and take $\varepsilon < b-a$ ; the function $W'_r$ being bounded on $[a+\varepsilon,b]$, one concludes that $W$ is absolutely continuous on  $[a+\varepsilon,b]$ (see e.g \cite[Chap. 5, Theorem 14]{Roy}). The function $W$ being bounded on  $[a+\varepsilon,b]$ and the function $x\mapsto e^{-x}$ being locally Lipschitz, one concludes that $f=e^{-W}$ is absolutely continuous on $ [a+\varepsilon,b]$. Thus it satisfies (see e.g \cite[Chap. 5, Corollary 15 ]{Roy}), 
\[
f(x) - f(a+\varepsilon) = \int_{a+\varepsilon}^{x} (-W'_r(u))f(u)\,du,\qquad \forall x \in [a+\varepsilon,b].
\]
Letting $\varepsilon\to 0$, one easily sees that 
\[
f(x)= \int_{a}^{x} (-W'_r(u))f(u)\,du,
\]
for $x \in [a,b]$ (the fact that the integrand is integrable on $[a,b]$ is also easily justified). According to \cite[Chap. 5, Theorem 14]{Roy} this shows that $f$ is absolutely continuous on any segment of the form $[a,b]$, $a<b$. Since $f$ vanishes on $(-\infty,a]$, it follows that $f$ is actually absolutely continuous on any segment.

Now, let us turn to the case $n\geq 2$. Let $W \in \mathcal{F}(\R^n)$ be essentially continuous, and consider the set 
\[
A = \{ x \in \partial \mathrm{dom}(W) : W(x)<\infty\}.
\] By assumption $\mathcal{H}^{n-1}(A)=0$, therefore $A_n = \{\bar{x} \in \R^{n-1} : \exists x_n \in \R, (\bar{x},x_n) \in A\}$ has Lebesgue measure $0$ (since it is the projection of $A$ onto $\R^{n-1}$). For all $\bar{x} \in \R^{n-1} \setminus A_n$, the function $x_n \mapsto W(\bar{x},x_n)$ is essentially continuous. According to the case $n=1$, one concludes that $F$ is absolutely continuous on any line of the form $\{\bar{x}\} \times \R$ with $\bar{x} \in \R^{n-1} \setminus A_n$. The same reasoning holds for lines parallel to the other axis.
\end{proof}

Now let us turn to the proof of Corollary \ref{cor:FMd1}.

\smallskip

\proof[Proof of Corollary \ref{cor:FMd1}.]
According to  \cite[Theorem 3]{FM08a}, the inequality $\IS_1(e)$ holds true: for all $f\in \mathcal{F}(\R)$ such that $\int e^{-f}\,dx >0$ and $\int e^{-f^*}\,dx>0$, it holds 
\begin{equation}\label{eq:FMd1}
\int e^{-f}\,dx\int e^{-f^*}\,dx \geq e.
\end{equation}
So the first part of Corollary \ref{cor:FMd1} is an immediate consequence of Theorem \ref{thm:mainresult} (note also that in dimension $1$ a convex function is essentially continuous if and only if it is continuous as a function taking values in $\R \cup\{+\infty\}$). 

Let us now show the optimality of the lower bound on $\delta_2$.
Define, for all $\eta_1,\eta_2$ satisfying the assumptions of Corollary \ref{cor:FMd1}, 
\[
\Delta(\eta_1,\eta_2) := \delta_{2} (\eta_1 \otimes \eta_2) -  \frac{1}{2}W_2^2(\nu_1,\nu_2)+\log (2\pi/ e).
\]
According to the proof of Theorem \ref{thm:mainresult} and \eqref{eq:IS-T3}, we see that if $V_1,V_2:\R \to \R$ then
\[
\Delta(\eta_1,\eta_2) = \mathcal{T}(\nu_1,\nu_2)-H(\eta_1|\Leb) - H(\eta_2|\Leb) -3.
\]
We will now consider sequences $(\eta_1^k)_{k\geq 1}$ and $(\eta_2^k)_{k\geq 1}$ approximating the two exponential probability measures $\tau$ and $\bar{\tau}$ defined by
\begin{equation}\label{eq:exp}
\tau(dx) = e^{-(1+x)} \mathbf{1}_{[-1,+\infty[}(x)\,dx\qquad \text{and}\qquad \bar{\tau}(dx) = e^{x-1} \mathbf{1}_{]-\infty,1]}(x)\,dx
\end{equation}
which are not admissible since their densities are not continuous. More precisely, let us  define $d\eta_1^k = \frac{1}{Z_1^k}e^{-V_1^k}\,dx$, where 
\[
V_1^k(x) = -k(x+1)\mathbf{1}_{]-\infty,-1[}(x)+(x+1)\mathbf{1}_{[-1,+\infty[}(x) 
\]
and $Z_1^k = \frac{1+k}{k}$ is the normalizing constant. We define similarly $\eta_2^k$ as the push forward of $\eta_1^k$ under the map $x\mapsto -x$.
A simple calculation shows that, for $i=1,2$, 
\[
H(\eta_i^k | \Leb) = -1 - \log\left(1 + \frac{1}{k}\right) \to -1
\]
as $k\to +\infty$. It is also not difficult to check that 
\[
\nu_1^k = \frac{1}{k+1} \delta_{-k}+  \frac{k}{k+1} \delta_1 \qquad \text{and}\qquad \nu_2^k = \frac{k}{k+1} \delta_{-1} + \frac{1}{k+1} \delta_{k}.
\]
The monotone optimal transport $\pi^k$ plan between $\nu_1^k$ and $\nu_2^k$ is given by
\[
\pi^k(-k,-1) = \frac{1}{k+1},\qquad \pi^k(1,-1) = \frac{k-1}{k+1},\qquad \pi^k(1,k) = \frac{1}{k+1}.
\]
So
\[
\mathcal{T}(\nu_1^k,\nu_2^k) = \int xy \,\pi^k(dxdy) = k \frac{1}{k+1} - \frac{k-1}{k+1}+ k \frac{1}{k+1}= 1.
\]
Therefore, 
\[
\Delta(\eta_1^k,\eta_2^k) = 2 \log\left(1 + \frac{1}{k}\right) \to 0
\]
as $k\to +\infty$. 
\endproof

\begin{remark}[Equality cases in $\IS_1(e)$]
Let $V,\bar{V} : \R \to \R \cup\{+\infty\}$ be the functions defined by $V (x) = x$ if $x \geq -1$ and $+\infty$ if $x<-1$ (resp. $\bar{V} (x) = -x$ if $x \leq 1$ and $+\infty$ if $x>1$).
As shown by Fradelizi and Meyer in \cite{FM08a}, the cases of equality in \eqref{eq:FMd1} are precisely the functions of the form $f(x)=V(ax)+b$, $a\neq 0$, $b\in \R$. 
As already mentioned in the proof, the probability measures $\tau$ and $\bar{\tau}$ defined by \eqref{eq:exp} are not admissible, because the functions $V$ and $\bar{V}$ are not continuous on $\R$. 
Note in particular that the moment measures $\nu_\tau$ and $\nu_{\bar{\tau}}$ associated to $\tau$ and $\bar{\tau}$ are respectively the Dirac masses $\delta_1$ and $\delta_{-1}$, which are not centered.
\end{remark}

\begin{remark}[Convergence of $\nu_i^k$, $i=1,2$]
Let us underline some subtleties concerning the convergence of the sequences $\nu_i^k$, $i=1,2$.
Note that $\nu_1^k$ is centered  for every $k\geq 1$ but weakly converges to $\delta_1$ which is not. This means that convergence is not true for the $W_1$ metric and a fortiori for the $W_2$ metric. This is confirmed by the fact that $\mathcal{T}(\nu_1^k,\nu_2^k) \to 1 \neq \mathcal{T}(\delta_1,\delta_{-1}) = -1$. Also, $\int x^2\,d\nu_i^k = k \to +\infty$ as $k\to +\infty$. Thus $W_2^2(\nu_1^k,\nu_2^k) = 2(k-1) \to +\infty$ as $k\to +\infty$. Therefore, the sequence 
\[
\delta_{2} (\eta_1^k \otimes \eta_2^k) -  \frac{1}{2}W_2^2(\nu_1^k,\nu_2^k)+\log (2\pi/ e)
\]
converges to $0$ but is the difference of two diverging sequences.
\end{remark}

\begin{remark}[Ghost equality cases]
Simple calculations show that
\[
H(\tau|\gamma_1)=H(\bar{\tau}|\gamma_1) = \frac{1}{2}\log\left(\frac{2\pi}{e}\right),\qquad W_2^2(\delta_1,\delta_{-1}) = 4,\qquad \text{and}\qquad \tilde{I}(\tau|\gamma_1) = \tilde{I}(\bar{\tau}|\gamma_1) = 2,
\]
where $\tilde{I}(\,\cdot\,|\gamma_1)$ is defined in \eqref{eq:tildeI}. Therefore the equation 
\[
H(\tau |\gamma_1)+ H(\bar{\tau} |\gamma_1) + \frac{1}{2}W_2^2(\delta_1,\delta_{-1}) = \frac{1}{2} \tilde{I}(\tau| \gamma_1)+ \frac{1}{2} \tilde{I}(\bar{\tau}| \gamma_1)+  \log (2\pi/ e)
\]
holds true. This suggests that the inequality
\[
H(\eta_1 |\gamma_1)+ H(\eta_2 |\gamma_1) + \frac{1}{2}W_2^2(\nu_1,\nu_2) \leq \frac{1}{2} \tilde{I}(\eta_1| \gamma_1)+ \frac{1}{2} \tilde{I}(\eta_2| \gamma_1)+  \log (2\pi/ e)
\] 
could perhaps be extended outside the domain of log-concave probability measures of the form $d\eta_i = e^{-V_i}\,dx$ with a continuous $V_i:\R \to \R\cup\{+\infty\}$. Nevertheless, the fact that the simple approximation scheme used in the proof of Corollary \ref{cor:FMd1} yields to blowing up quantities seems to leave little hope for that.
\end{remark}

Let us now turn to the proof of Theorem \ref{thm:mainresult2}.
\smallskip

\proof[Proof of Theorem \ref{thm:mainresult2}.]
According to Fradelizi-Meyer \cite[Theorem 10]{FM08a}, the inequality $\IS_{n,u}(4)$ holds true. Therefore, Theorem \ref{thm:mainresult} yields to the following reinforcement of the Gaussian logarithmic Sobolev inequality: if $\eta_1,\eta_2$ are unconditional log-concave probability measures on $\R^n$ such that, for $i=1,2$, $d\eta_i = e^{-V_i}\,dx$ with $V_i: \R^n\to \R\cup\{+\infty\}$ an essentially continuous convex function, it holds
\[
H(\eta_1 |\gamma_n)+ H(\eta_2 |\gamma_n)+\frac{1}{2}W_2^2(\nu_1,\nu_2)  \leq n\log (\pi/ 2) +   \frac{1}{2} I(\eta_1| \gamma_n)   + \frac{1}{2} I(\eta_2| \gamma_n) ,
\]
where, for $i=1,2$,  $\nu_i$ is the moment measure of $\eta_i.$

Consider the symmetric exponential probability measure $\tau_s(dx) = \frac{1}{2}e^{-|x|}\,dx$ and let us choose $\eta_2(dx) = \tau_s^{\otimes n}(dx) =  \frac{1}{2^n} e^{-\sum_{i=1}^n |x_i|}\,dx$ (whose minus log density realizes the equality case in $\IS_{n,u}(4)$).
Then simple calculations show that $\nu_2 = \left(\frac{1}{2}\delta_{-1} + \frac{1}{2}\delta_1\right)^{\otimes n} = \lambda_{C_n}$,
\[
H(\tau_s^{\otimes n} |\gamma_n) = \frac{n}{2} \log\left(\frac{e\pi}{2}\right) \qquad \text{and}\qquad I(\tau_s^{\otimes n}| \gamma_n) = n.
\]
Therefore, for any $\eta(:=\eta_1)$ as above, one gets
\[
H(\eta |\gamma_n)+\frac{1}{2}W_2^2\left(\nu,  \lambda_{C_n}\right)  \leq \frac{n}{2}\log \left(\frac{\pi e}{2}\right)+   \frac{1}{2} I(\eta| \gamma_n).
\]
Consider now the sequence of probability measures $(\eta_k)_{k\geq 1}$ given by $d\eta_k = \frac{1}{Z_k}e^{-V_k}\,dx$, with 
\[
V_k(x) = \left\{\begin{array}{ll} k|x-1| & \text{if } x\geq 1  \\ 0 & \text{if } x\in [-1,1]  \\ k|x+1| & \text{if } x \leq -1 \end{array}\right.
\]
and $Z_k = \frac{2(k+1)}{k}$. Easy calculations show that, when $k\to +\infty$,
\[
H(\eta_k | \gamma_1) = \frac{1}{2}\log \left(\frac{\pi}{2}\right) - \log \left(1+ \frac{1}{k}\right) +  \frac{k}{(k+1)} \left[ \frac{1}{6} + \frac{2}{k^3} + \frac{2}{k^2}\right] = \frac{1}{2}\log \left(\frac{\pi}{2}\right) + \frac{1}{6}  + o(1),
\]
\[
I(\eta_k | \gamma_1) = \frac{k}{k+1} \left[ \frac{1}{3} + \frac{1}{k}\left(\frac{1}{k^2} + \left(1-k + \frac{1}{k}\right)^2\right)\right] = \frac{1}{3} + k-3 + o(1)
\]
and 
\[
W_2^2\left(\nu_k, \frac{1}{2}\delta_{-1} + \frac{1}{2}\delta_1\right) = \frac{1}{(k+1)} \left(k^2-k+1\right) = k-2 + o(1),
\]
where $\nu_k = \frac{1}{2(k+1)} \delta_{-k} + \frac{k}{k+1} \delta_0 + \frac{1}{2(k+1)} \delta_{k}$ is the moment measure of $\eta_k.$
So, 
\[
 \frac{1}{2} I(\eta_k| \gamma_1) - H(\eta_k | \gamma_1) - \frac{1}{2} W_2^2\left(\nu_k, \frac{1}{2}\delta_{-1} + \frac{1}{2}\delta_1\right) = -\frac{1}{2}\log \left(\frac{\pi e}{2}\right) + o(1).
\]
Since 
\[
 \frac{1}{2} I(\eta_k^{\otimes n}| \gamma_n) - H(\eta_k^{\otimes n} | \gamma_n) - \frac{1}{2} W_2^2\left(\nu_k^{\otimes n}, \lambda_{C_n}\right) = n\left[ \frac{1}{2} I(\eta_k| \gamma_1) - H(\eta_k | \gamma_1) - \frac{1}{2} W_2^2\left(\nu_k, \frac{1}{2}\delta_{-1} + \frac{1}{2}\delta_1\right)\right],
\]
this completes the proof. 
\endproof

\section{Proof of Theorem \ref{thm:dualitygen}}\label{sec:proof}
During the proof, we will use the following version of the min-max theorem due to Sion \cite{Sion58}.
\begin{theorem}[Sion min-max theorem]\label{thm:sion}
Let $\X$ and $\Y$ be two convex subsets of some linear topological spaces. 
Let $F:\X\times \Y \to \R$ be such that $f(x,\cdot)$ is concave and upper semicontinuous for every $x \in \X$ and $f(\cdot, y)$ is convex and lower semicontinuous for every $y \in \Y$. If $\X$ or $\Y$ is compact, then 
\[
\inf_{x \in \X} \sup_{y \in \Y} F(x,y) = \sup_{y \in \Y}  \inf_{x \in \X} F(x,y).
\]
\end{theorem}

\smallskip

\proof[Proof of Theorem \ref{thm:dualitygen}.] Let $m$ be a log-concave measure and $f\in \mathcal{F}(\R^n)$ some convex function such that $0<\int e^{-f^*}\,dm <+\infty$ (according to Proposition \ref{prop:dualtriv}, there is nothing to prove when this integral is $+\infty$).

\noindent \textit{First step.} 
By definition of $K(\,\cdot\,|m)$, it holds 
\begin{align*}
\sup_{\nu \in \mathcal{P}_1(\R^n)} \left\{ \int (-f)\,d\nu - K(\nu |m)\right\} & = \sup_{\nu \in \mathcal{P}_1(\R^n)} \inf_{\varphi \in \widetilde{\mathcal{F}}_{\mathrm{Lip}}(\R^n)} \left\{\int (\varphi-f)\,d\nu - \log \int e^{-\varphi^*}\,dm\right\},
\end{align*}
where we recall that $\widetilde{\mathcal{F}}_{\mathrm{Lip}}(\R^n)$ is defined just after Proposition \ref{prop:dualtriv}.
Let us assume for a moment that $f$ is such that
\begin{equation}\label{eq:minmax}
\sup_{\nu \in \mathcal{P}_1(\R^n)} \inf_{\varphi \in \widetilde{\mathcal{F}}_{\mathrm{Lip}}(\R^n)} \left\{\int (\varphi-f)\,d\nu - \log \int e^{-\varphi^*}\,dm\right\} = \inf_{\varphi \in  \widetilde{\mathcal{F}}_{\mathrm{Lip}}(\R^n)} \sup_{\nu \in \mathcal{P}_1(\R^n)} \left\{\int (\varphi-f)\,d\nu - \log \int e^{-\varphi^*}\,dm\right\}.
\end{equation}
This interversion of $\inf$ and $\sup$ will be justified in the second step below.
Let us show that 
\[
\inf_{\varphi \in  \widetilde{\mathcal{F}}_{\mathrm{Lip}}(\R^n)} \sup_{\nu \in \mathcal{P}_1(\R^n)} \left\{\int (\varphi-f)\,d\nu - \log \int e^{-\varphi^*}\,dm\right\} = -\log \int e^{-f^*}\,dm.
\]
Note that 
\[
\sup_{\nu \in \mathcal{P}_1(\R^n)} \int (\varphi-f)\,d\nu = \sup_{x\in \R^n} (\varphi(x) - f(x)):= m_\varphi.
\]
So, 
\begin{align*}
\inf_{\varphi \in  \widetilde{\mathcal{F}}_{\mathrm{Lip}}(\R^n)} \sup_{\nu \in \mathcal{P}_1(\R^n)} \left\{\int (\varphi-f)\,d\nu - \log \int e^{-\varphi^*}\,dm\right\}  &= \inf_{\varphi \in  \widetilde{\mathcal{F}}_{\mathrm{Lip}}(\R^n)} \left\{ m_\varphi- \log \int e^{-\varphi^*}\,dm\right\}\\
& =  \inf_{\varphi \in  \widetilde{\mathcal{F}}_{\mathrm{Lip}}(\R^n) \text{ s.t } m_\varphi=0} \left\{- \log \int e^{-\varphi^*}\,dm\right\}.
\end{align*}
Let us show that 
\begin{equation}\label{eq:step1}
\inf_{\varphi \in  \widetilde{\mathcal{F}}_{\mathrm{Lip}}(\R^n) \text{ s.t } m_\varphi=0} \left\{- \log \int e^{-\varphi^*}\,dm\right\}=  -\log \int e^{-f^*}\,dm.
\end{equation}
First, note that if $\varphi$ is such that $m_{\varphi}=0$, then $\varphi \leq f$ and so $-\log \int e^{-\varphi^*}\,dm \geq -\log \int e^{-f^*}\,dm$. Conversely, let us construct a sequence of convex and Lipschitz functions $f_k$ such that $m_{f_k}=0$ and $\int e^{-f_k^*}\,dm \to  \int e^{-f^*}\,dm$. The function $f$ being convex, one can find $a\in \R^n$ and $b\in \R$ such that $f(x) \geq a\cdot x + b$, $x\in \R^n.$ Let us denote by $g(x) = f(x)-(a\cdot x + b)$, which is convex and non-negative. Consider the sequence of convex functions $g_k$ defined by
\[
g_k(x) = \inf_{y\in \R^n}\{g(y) +k |x-y|\}, \qquad x\in \R^n, k\geq 1,
\] 
as in Lemma \ref{lem:infconv}, which is such that $g_k \leq g$, $g_k$ is $k$-Lipschitz, and $g_k^* = g^* + \chi_{B_k}$. Letting $f_k(x) = g_k(x) + a\cdot x + b$, one gets that 
\[
f_k^*(y) = g_k^*(y-a)-b = g^*(y-a)+\chi_{B_k}(y-a)-b = f^*(y)+\chi_{B_k}(y-a).
\]
Therefore, $\int e^{-f_k^*}\,dm \to  \int e^{-f^*}\,dm$, by the monotone convergence theorem (and in particular $f_k$ belongs to $ \widetilde{\mathcal{F}}_{\mathrm{Lip}}(\R^n)$ for all $k$ large enough). Note that $m_{f_k} = \sup_{x\in \R^n} \{g_k(x) - g(x)\} \leq 0$. Since $g$ is bounded from below and lower semi-continuous, it reaches its infimum at some point $\alpha \in \R^n$, and it is easily seen that $g_k(\alpha) = g(\alpha)$. Therefore, $m_{f_k} = 0$, which completes the proof of \eqref{eq:step1}.

\noindent \textit{Second step.} In this step, we show that if $f\in \mathcal{F}(\R^n)$ is such that $0<\int e^{-f^*}\,dm <+\infty$ and such that $D:=\mathrm{dom}(f)$ is compact and $f$ is bounded on $D$, then \eqref{eq:minmax} holds true. Let us denote by $\mathcal{P}(D)$ the set of Borel probability measures on $D$ and  
consider the function $F : \mathcal{P}(D) \times \widetilde{\mathcal{F}}_{\mathrm{Lip}}(\R^n) \to \R \cup \{-\infty\}$ defined by
\[
F(\nu,\varphi) =\int (\varphi-f)\,d\nu - \log \int e^{-\varphi^*}\,dm.
\]
Let us denote by 
\[
C(f) = \sup_{\nu \in \mathcal{P}_1(\R^n)} \left\{ \int (-f)\,d\nu - K(\nu |m)\right\} = \sup_{\nu \in \mathcal{P}(D)} \inf_{\varphi \in\widetilde{\mathcal{F}}_{\mathrm{Lip}}(\R^n) } F(\nu,\varphi)
\] 
and note that this quantity is finite according to Proposition \ref{prop:dualtriv}.
Let us equip $\mathcal{P}(D)$ with the usual weak topology. Since $D$ is compact, it follows from Prokhorov theorem that $\mathcal{P}(D)$ is also compact. 
Let us denote by $\mathcal{M}(D)$ the linear space of all finite Borel signed measures $\nu$ on $D$, and equip it with the coarsest topology that makes continuous the functionals $\mathcal{M}(D)\ni \nu \mapsto \int \varphi\,d\nu$, for all continuous function $\varphi$ on $D$. In restriction to  $\mathcal{P}(D)$, this topology coincides with the weak topology. Therefore, $\X:= \mathcal{P}(D)$ can be seen as a compact convex subset of $\mathcal{M}(D)$.

Consider the space $\mathcal{C}(\R^n)$ of all continuous functions on $\R^n$ and equip it with the topology of uniform convergence over all compact subsets of $\R^n$. The set $\Y:=\widetilde{\mathcal{F}}_{\mathrm{Lip}}(\R^n)$ is a convex subset of $\mathcal{C}(\R^n)$. Indeed, $\widetilde{\mathcal{F}}_{\mathrm{Lip}}(\R^n) =  \{ \varphi \in \mathcal{F}_{\mathrm{Lip}}(\R^n) : - \log \int e^{-\varphi^*}\,dm <+\infty\}$ and this set is convex thanks to Lemma \ref{lem:CEK}.

With these notations, it follows from what precedes that
\[
C(f) =  \sup_{\nu \in \X}  \inf_{\varphi \in \Y }  F(\nu,\varphi).
\]
In order to permute $\inf$ and $\sup$, let us check the assumptions of Theorem \ref{thm:sion}.
\begin{itemize}
\item Restricted to $\X \times \Y$, the functional $F$ takes finite values. Indeed, since $f$ is bounded on $D$, it follows that $\int |\varphi-f|\,d\nu <+\infty$ for all $\nu \in \X$ and $\varphi \in \mathcal{F}_{\mathrm{Lip}}(\R^n)$. Furthermore, if $\varphi \in \widetilde{\mathcal{F}}_{\mathrm{Lip}}(\R^n)$, then $\varphi^* = +\infty$ outside a closed ball, and so $\int e^{-\varphi^*}\,dm <+\infty$ (and $\neq 0$ by definition of $\widetilde{\mathcal{F}}_{\mathrm{Lip}}(\R^n)$). 

\item For any fixed $\varphi \in \Y$, the map $\X \ni \nu \mapsto F(\nu,\varphi)$ is upper-semicontinuous (this follows from the lower semicontinuity and boundedness of $f$ and Portmanteau theorem).

\item For any fixed $\nu \in \X$, the map $\Y \ni \varphi \mapsto F(\nu,\varphi)$ is lower semi-continuous. Indeed, the map $\Y \ni \varphi \mapsto \int \varphi \,d\nu$ is clearly continuous since $\nu\in \X$ has a compact support. Furthermore, if $\varphi_k$ is a sequence of elements of $\Y$ converging to some $\varphi \in \Y$, then we claim that
\begin{equation}\label{eq:limsup}
\limsup_{k\to \infty} \int e^{-\varphi_k^*} \,dm \leq  \int e^{-\varphi^*} \,dm,
\end{equation}
which gives the announced lower-semicontinuity.
To prove \eqref{eq:limsup}, we slightly adapt an argument from the proof of \cite[Lemma 17]{CEK15}. Since $m$ is log-concave, there exists $\alpha >0$ such that $\int e^{-\alpha |x|}\,dm<+\infty$. For any $r>0$, denote by 
\[
\psi_r(y) = \sup_{|x|\leq r} \{x\cdot y - \varphi(x)\},\qquad x\in \R^n.
\]
Then $\psi_r$ converges to $\varphi^*$ monotonically, as $r\to \infty$, and $\psi_\alpha(y) \geq \alpha |y| - M$, where $M = \sup_{|x|\leq \alpha} \varphi(x)$.
So, using the dominated convergence theorem,
\[
\int e^{-\psi_r}\,dm \to \int e^{-\varphi^*}\,dm
\]
as $r \to \infty$. Take some $\varepsilon>0$, and $r_0\geq \alpha$ large enough so that $\int e^{-\psi_{r_0}}\,dm \leq   \int e^{-\varphi^*}\,dm + \varepsilon$.
Define $\psi_{r_o}^k (y)= \sup_{|x|\leq r_0} \{x\cdot y - \varphi_k(x)\}$, $y\in \R^n.$ Since $\varphi_k$ converges uniformly to $\varphi$ on any compact set, one sees that $\psi_{r_o}^k(y) \to \psi_{r_o}(y)$ for all $y\in \R^n.$ Furthermore, $M' := \sup_{k\geq 1} \sup_{|x|\leq \alpha} \varphi(x) <+\infty$ and so $\psi_{r_o}^k (y) \geq \alpha |y| - M'$, $y\in \R^n$. Therefore, by the dominated convergence theorem
\[
\int e^{-\psi_{r_0}^k}\,dm \to \int e^{-\psi_{r_0}}\,dm \leq \int e^{-\varphi^*}\,dm + \varepsilon.
\]
Since, $\int e^{-\varphi_k^*}\,dm \leq \int e^{-\psi_{r_0}^k}\,dm$, one concludes that
\[
\limsup_{k\to \infty} \int e^{-\varphi_k^*} \,dm \leq  \int e^{-\varphi^*} \,dm + \varepsilon
\]
which gives \eqref{eq:limsup} by letting $\varepsilon\to0.$

\item Finally, for any fixed $\varphi \in \Y$, the map $\nu \ni \X \mapsto F(\nu,\varphi)$ is concave (and even linear), and according to Lemma \ref{lem:CEK}, for any fixed $\nu \in \X$, the map $\Y\ni \varphi\mapsto F(\nu,\varphi)$ is convex. 
\end{itemize}
Therefore, applying Theorem \ref{thm:sion}, one gets that
\[
C(f) = \inf_{\varphi \in \widetilde{\mathcal{F}}_{\mathrm{Lip}}(\R^n) } \sup_{\nu \in \mathcal{P}(D)}  F(\nu,\varphi) = \inf_{\varphi \in \widetilde{\mathcal{F}}_{\mathrm{Lip}}(\R^n) } \sup_{\nu \in \mathcal{P}_1(\R^n)}  F(\nu,\varphi).
\]

\noindent \textit{Third step.} According to the two preceding steps, the equality
\[
\sup_{\nu \in \mathcal{P}_1(\R^n)} \left\{ \int (-f)\,d\nu - K(\nu |m)\right\} = - \log \int e^{-f^*}\,dm
\]
holds true for any function $f \in \mathcal{F}(\R^n)$ such that $0<\int e^{-f^*}\,dm <\infty$ and such that $D:=\mathrm{dom}(f)$ is compact and $f$ is bounded on $D$.
Let us finally remove this last assumption. Consider $f\in \mathcal{F}(\R^n)$ such that $0<\int e^{-f^*}\,dm <\infty$. For all $k\geq 1$, define $D_k= \{f\leq k\} \cap B_k$, $k\geq 1$ and  $f^k = f+\chi_{D_k}$, where $B_k$ is the closed ball of radius $k$ centered at $0$. The lower semicontinuity of $f$ implies that the sets $D_k$, $k\geq 1$, are compact. The sequence $f^k$, $k\geq 1$, being non increasing, it follows that the sequence $(f^k)^*$, $k\geq 1$, is non decreasing. Moreover, for any $y\in \R^n$,
\[
\sup_{k\geq 1} (f^k)^*(y) = \sup_{k\geq 1} \sup_{x\in \R^n} \{x\cdot y - f(x) - \chi_{D_k}(x)\} = \sup_{x\in \R^n}  \sup_{k\geq 1}  \{x\cdot y - f(x) - \chi_{D_k}(x)\} = f^*(y).
\] 
Let us admit for a moment that $0<\int e^{-(f^k)^*(y)}\,m(dy)<+\infty$, for all $k$ large enough. 
Letting $k \to \infty$ in the identity
\[
- \log \int e^{-(f^k)^*(y)}\,m(dy) = \sup_{\nu \in \mathcal{P}_1(\R^n)} \left\{\int (-f^k)\,d\nu -K(\nu|m)\right\}
\]
and reasoning as in the end of the proof of Theorem \ref{thm:dualLeb}, one concludes that the identity holds for $f$ as well.
To finish the proof, let us show that $0<\int e^{-(f^k)^*(y)}\,m(dy)<+\infty$ for all $k$ large enough. Since $(f^k)^* \leq f^*$, it is clear that $0<\int e^{-(f^k)^*(y)}\,m(dy)$ for all $k\geq 1$. So, according to Lemma \ref{lem:elem}, $\int e^{-(f^k)^*(y)}\,m(dy)<+\infty$ if and only if $0$ belongs to the interior of $\mathrm{dom} ( ((f^k)^*+V)^* )$. Note that 
\[
((f^k)^*+V)^*(x) = f^k \square V^*(x) := \inf_{y\in \R^n} \{f^k(y) + V^*(x-y)\},
\]
where $\square$ denotes the infimum convolution operations. From this follows easily that
\[
\mathrm{dom} ( ((f^k)^*+V)^* ) = \mathrm{dom}(f^k) + \mathrm{dom} (V^*) = (\mathrm{dom}(f)\cap D_k) + \mathrm{dom} (V^*).
\]
Since $0<\int e^{-f^*}\,dm <+\infty$, we know that $0$ belongs to the interior of $\mathrm{dom}(f)+ \mathrm{dom} (V^*).$ Therefore, there is some $\ep>0$ such that $\ep [-1,1]^n \subset \mathrm{dom}(f)+ \mathrm{dom} (V^*)$. So, for any $u \in \{-1,1\}^n$, there exist $a_u \in \mathrm{dom}(f)$ and $b_u \in \mathrm{dom} (V^*)$ such that $a_u + b_u = \ep u$. Choose $k_o$ large enough so that the $2^n$ points $a_u$, $u\in \{-1,1\}^n$, all belong to $\mathrm{dom}(f)\cap D_{k_o}$. Then, for all $k \geq k_o$, the convex set $(\mathrm{dom}(f)\cap D_k) + \mathrm{dom} (V^*)$ contains the family of points $\ep u$, $u \in \{-1,1\}^n$ and so it contains their convex hull $\ep [-1,1]^n.$ This proves that $0$ belongs to the interior of $\mathrm{dom} ( ((f^k)^*+V)^* )$ and completes the proof.
\endproof

\section{Yet another proof of Theorem \ref{thm:dualLeb}.}\label{Sec:Lebesgue} In this section, we indicate another way, based on a general Fenchel-Moreau biconjugation theorem, to prove Theorem \ref{thm:dualLeb}. The same method could be used to establish Theorem \ref{thm:dualitygen} as well, but we prefer to restrict to the case where $m$ is the Lebesgue measure to avoid lengthy developments. 

Let $\Omega \subset \R^n$ be an open subset and denote by $\mathcal{C}(\Omega)$ the space of all continuous functions on $\Omega$. We will equip $\mathcal{C}(\Omega)$ with the topology of uniform convergence on compact sets of $\Omega$. This is the topology generated by the collection of seminorms $p_{K_i}$, $i\geq 1$, defined by 
\[
p_{K_i} (f) =\sup_{x\in K_i} |f(x)|,\qquad f\in \mathcal{C}(\Omega),
\]
where $(K_i)_{i\geq 1}$ is an increasing sequence of compact sets such that $\Omega = \cup_{i\geq 1} K_i$. 
The following result is a consequence of the Riesz-Markov representation theorem (see \cite[Proposition 14 page 156]{Bou}).
\begin{theorem}
The topological dual space $(\mathcal{C}(\Omega))'$ of $\mathcal{C}(\Omega)$ can be identified with the set of finite signed Borel measures $\mu$ with a compact support $K \subset \Omega$.
\end{theorem}
%\begin{proof}
%We sketch the proof for the sake of completeness. Let $\ell: \mathcal{C}(\Omega) \to \R$ be a continuous linear functional. Since $\ell$ is continuous, there exists an open neighborhood  $A $ of $0_\Omega$ (the zero function on $\Omega$) such that $|\ell(f)| \leq 1$ for any $f \in A$. By definition of the topology of $\mathcal{C}(\Omega)$, there exists $i_o\geq 1$ and $\ep>0$ such that $\{f : p_{K_{i_o}}(f)\leq \ep\} \subset A$. Therefore, 
%\[
%|\ell(f)| \leq \frac{1}{\ep} p_{K_{i_o}}(f), \qquad f\in \mathcal{C}(\Omega).
%\]
%If $g : K_{i_o} \to \R$ is a continuous function on $K_{i_o}$ and $f_1,f_2 : \Omega \to \R$ are continuous functions such that $f_1=f_2=g$ on $K_{i_o}$, one easily concludes from this that $\ell(f_1) = \ell(f_2)$. 
%Therefore one can define $\tilde{\ell}(g) = \ell(f)$, for any $f \in \mathcal{C}(\Omega)$ such that $f=g$ on $K_{i_o}$. This map $\tilde{\ell}$ is a continuous linear functional on $\mathcal{C}(K_{i_o})$ (equipped with $p_{K_{i_o}}$) and it holds $\ell(f) = \tilde{\ell}(f_{\vert K_{i_o}})$, for all $f\in \mathcal{C}(\Omega)$. According to the Riesz-Markov representation theorem, there exists a finite signed Borel measure $\mu$ on $\R^n$ such that $|\mu|(K_{i_o}^c) = 0$ such that 
%\[
%\tilde{\ell}(g) = \int g \,d\mu,\qquad \forall g \in \mathcal{C}(K_{i_o}).
%\]
%and so, for any $f \in \mathcal{C}(\Omega)$, $\ell(f) = \int f \,d\mu$ which completes the proof.
%\end{proof}

Now let us define the conjugate operation on $\mathcal{C}(\Omega)$. For any $f \in \mathcal{C}(\Omega)$, let $c_{\Omega}(f)$ be the function defined on $\R^n$ as follows
\[
c_\Omega(f) (y)= \sup_{x\in \Omega} \{x\cdot y - f(x)\},\qquad y\in \R^n.
\]
We also define the functional $\Lambda_{\Omega} : \mathcal{C}(\Omega) \to \R \cup \{\pm\infty\}$ as follows
\[
\Lambda_{\Omega} (f) = - \log \int e^{-c_\Omega(f)}\,dx.
\]

\begin{lemma}\label{lem:sciproper}
If $0 \in \Omega$, the functional $\Lambda_{\Omega}$ is lower semi-continuous, convex and never takes the value $-\infty$.
\end{lemma}
\begin{proof}
The convexity of $\Lambda_\Omega$ follows from the log-concavity of the Lebesgue measure exactly as in Lemma \ref{lem:CEK}.
Let $a \in \Omega$ and $r_o>0$ small enough so that $B_{r_o} \subset \Omega$. Then if $f\in \mathcal{C}(\Omega)$, then denoting by $M = \sup_{x\in B_{r_o}} f(x)$, it holds
\[
c_{\Omega}(f)(y) \geq \sup_{x \in B_{r_o}} \{x\cdot y\} - M = r_o |y| - M.
\]
Therefore, $\int e^{-c_\Omega(f)}\,dx<+\infty$ and so $\Lambda_\Omega(f)>-\infty$, for all $f \in \mathcal{C}(\Omega)$.
Reasoning as in the proof of Theorem \ref{thm:dualitygen} (more precisely, the proof of \eqref{eq:limsup}, taking $\alpha=r_o$), one sees that if $(f_n)_{n\geq 1}$ is a sequence of elements of $\mathcal{C}(\Omega)$ converging to $f \in \mathcal{C}(\Omega)$ (uniformly on any compact of $\Omega$), then 
\[
\limsup_{n \to \infty} \int e^{-c_\Omega(f_n)}\,dx \leq \int e^{-c_\Omega(f)}\,dx,
\]
which gives the announced lower semicontinuity of $\Lambda_{\Omega}$.
\end{proof}
We recall the following general version of the Fenchel-Moreau duality theorem (see for instance \cite[Theorem 2.3.3]{Zal02}).
\begin{theorem}[General Fenchel-Moreau theorem]\label{Fenchel-Legendre}
Let $E$ be a Hausdorff locally convex topological vector space and $E'$ its topological dual space. For any lower semicontinuous convex function $F \colon E\to ]-\infty,\infty]$, it holds
\[
F(x)=\sup_{\ell \in E'}\{ \ell(x) - F^*(\ell)\},\qquad  x\in E,
\]
where the Fenchel-Legendre transform $F^*$ of $F$ is defined by 
\[
F^*(\ell)=\sup_{x\in E}\{\ell(x) -F(x)\},\qquad  \ell \in E'.
\]
\end{theorem}
We are now ready to give the alternative proof of Theorem \ref{thm:dualLeb}.

\smallskip
\proof[Alternative proof of Theorem \ref{thm:dualLeb}.]
Let $f \in \mathcal{F}(\R^n)$ be such that $\int e^{-f^*}\,dx>0$ and denote by  $\Omega$ the interior of $\mathrm{dom}(f)$ (possibly empty).

If $0$ does not belong to $\Omega$, then according to Lemma \ref{lem:elem}, $\int e^{-f^*}\,dx = +\infty$. Applying Proposition \ref{prop:dualtriv} gives the announced equality.

Now let us assume that $0 \in \Omega$. Since $f$ is convex, $f$ is continuous on $\Omega$ and so $f_{\vert \Omega} \in \mathcal{C}(\Omega).$ Moreover, since $f$ is lower semicontinuous, it holds $c_{\Omega}(f_{\vert \Omega}) = f^*$ (the values of $f$ on the boundary of $\mathrm{dom}(f)$ are fully determined by the values of $f$ on $\Omega$). So applying, Theorem \ref{Fenchel-Legendre} to $\Lambda_\Omega$ (and $E = \mathcal{C}(\Omega)$) yields to
\[
- \log \int e^{-f^*}\,dx = \sup_{\mu} \left\{\int f\,d\mu - \Lambda_{\Omega}^*(\mu)\right\},
\]
where the supremum runs over the set of all finite signed measures $\mu$ with a compact support in $\Omega$, and 
\[
\Lambda_{\Omega}^*(\mu) = \sup_{\varphi \in \mathcal{C}(\Omega)} \left\{\int \varphi \,d\mu + \log \int e^{-c_{\Omega}(\varphi)}\,dx\right\}.
\]

We claim that $\Lambda_{\Omega}^*(\mu)=+\infty$ if $\mu$ is not of the form $\mu = - \nu$ with $\nu$ a probability measure. Indeed, let $\mu = \mu^+ - \mu^-$ be the Hahn decomposition of $\mu$ as a difference of finite positive measures, and assume that $\mu^+(\Omega)>0$. Then there is at least one compactly supported function $\psi_o : \Omega \to \R_+$ such that $\int \psi_o \,d\mu^+ >0$. By construction of $\mu^+$, it holds $\int \psi_o\,d\mu^+ = \sup \{\int \varphi \,d\mu : 0\leq \varphi\leq \psi_o\}$, so we conclude that there exists at least one compactly supported function $\varphi_o:\Omega \to \R_+$ such that $\int \varphi_o\,d\mu>0$. For all $t>0$, choosing $\varphi(x) = t\varphi_o(x) + |x|$, $x\in \Omega$, as test function yields to
\begin{align*}
\Lambda_{\Omega}^*(\mu)  &\geq \int t\varphi_o(x)+|x|\,\mu(dx) + \log \int e^{-c_{\Omega}(t\varphi_o+|\,\cdot\,|)}\,dx\\
& \geq  \int t\varphi_o(x)+|x|\,\mu(dx) + \log \int e^{-c_{\Omega}(|\,\cdot\,|)}\,dx,
\end{align*}
where the second inequality comes from the monotonicity property of $c_\Omega$ : $h\leq g \Rightarrow c_\Omega(h)\geq c_{\Omega}(g)$. It is easily checked that $\int e^{-c_{\Omega}(|\,\cdot\,|)}\,dx \neq 0$ and so, letting $t\to \infty$, gives that $\Lambda_{\Omega}^*(\mu) =+\infty.$ Finally, replacing $\varphi$ by $\varphi + u$, $u\in \R$, in the definition of $\Lambda_{\Omega}^*(\mu)$, and using that $c_{\Omega}(\varphi+u) = c_{\Omega}(\varphi)- u$, one gets
\[
\Lambda_{\Omega}^*(\mu) = \sup_{\varphi \in \mathcal{C}(\Omega)} \sup_{u\in \R} \left\{\int \varphi \,d\mu + \log \int e^{-c_{\Omega}(\varphi)}\,dx + u(\mu(\Omega)+1)\right\},
\]
which shows that $\Lambda_{\Omega}^*(\mu)=+\infty$ if $\mu(\Omega) \neq -1$.

Finally, let us fix some probability measure $\nu$ having a compact support in $\Omega$ and let us show that $\Lambda_{\Omega}^*(-\nu)=K(\nu|\Leb)$.
Suppose that $\varphi \in \mathcal{F}_{\mathrm{Lip}}(\R^n)$, then $c_{\Omega} (\varphi) \leq \varphi^*$ and so 
\[
\Lambda_{\Omega}^*(-\nu) \geq \sup_{\varphi \in \mathcal{F}_{\mathrm{Lip}}(\R^n)} \left\{ \int -\varphi \,d\nu + \log \int e^{-\varphi^*}\,dx\right\} = K(\nu|\Leb).
\]
Let us show the converse inequality. 
Let $g \in \mathcal{C}(\Omega)$ and let $K$ denote the convex hull of the support of $\nu$. Consider the function $h= g+\chi_K$. Since $K\subset \Omega$, it holds
\[
h^*(y) = c_K(g)(y):= \sup_{x\in K}\{x\cdot y - g(x)\} \leq c_{\Omega}(g)(y),\qquad \forall y\in \R^n.
\] 
Consider the function $\varphi:\R^n\to \R\cup \{+\infty\}$ defined by $\varphi = h^{**}$.  The function $\varphi$ belongs to the class $\mathcal{F}(\R^n)$ and is such that $\varphi \leq h$ (it is actually the convex enveloppe of $h$, that is to say the greatest convex function below $h$). In particular $\varphi \in L^1(\nu)$ and it holds
\begin{align*}
\int -g\,d\nu + \log \int e^{-c_\Omega(g)}\,dx &\leq \int -h\,d\nu + \log \int e^{-c_K(g)}\,dx\\
  &\leq \int -\varphi\,d\nu + \log \int e^{-c_K(g)}\,dx\\&=  \int -\varphi\,d\nu + \log \int e^{-\varphi^*}\,dx,
\end{align*}
where the last equality comes from the fact that $\varphi^*=h^{***}= h^* = c_K(g)$.  We conclude from this that $\Lambda_{\Omega}^*(-\nu) \leq K(\nu|\Leb)$, which completes the proof. 
\endproof

\noindent \textbf{Acknowledgments:} The author thanks Max Fathi and Matthieu Fradelizi for useful discussions and valuable comments during the preparation of this work. He also wants to thank the two anonymous referees for their careful reading and their suggestions that improved the quality of this work. The author is supported by a grant of the Simone and Cino Del Duca Foundation. This research has been conducted within the FP2M federation (CNRS FR 2036)
%\bibliographystyle{amsplain}
%\bibliography{bib}

\providecommand{\bysame}{\leavevmode\hbox to3em{\hrulefill}\thinspace}
\providecommand{\href}[2]{#2}

\end{document}